\documentclass[%
onefignum,onetabnum]{siamart171218}

\usepackage{lipsum}
\usepackage{amsfonts}
\usepackage{graphicx}
\usepackage{epstopdf}

\usepackage{algorithm}
\usepackage{algpseudocode}
\usepackage{tikz}
\usepackage{todonotes}

\usepackage{hyperref}
\usepackage{booktabs}

\usepackage{tabularx}
\usepackage{lipsum}

\usepackage[percent]{overpic}

\usepackage[subpreambles=true]{standalone}

\ifpdf
  \DeclareGraphicsExtensions{.eps,.pdf,.png,.jpg}
\else
  \DeclareGraphicsExtensions{.eps}
\fi

\newsiamremark{remark}{Remark}
\newsiamremark{hypothesis}{Hypothesis}
\crefname{hypothesis}{Hypothesis}{Hypotheses}
\newsiamthm{claim}{Claim}

\usepackage[toc,page]{appendix}

\newcommand{\shorttitle}{Convergence analysis of pixel-driven transforms}
\newcommand{\titlename}{Convergence analysis of pixel-driven Radon and fanbeam transforms}

\headers{\shorttitle}{Kristian Bredies and Richard Huber}

\title{\titlename \thanks{November 4, 2020. %
    \funding{
      International Research Training Group %
      ``Optimization and Numerical Analysis for Partial Differential Equations with Nonsmooth Structures'', funded by the German Research Council (DFG) and
      the Austrian Science Fund (FWF)
(grant W1244).
  }}}

\author{Kristian Bredies%
  \and Richard Huber\thanks{Institute of Mathematics and Scientific Computing,
    Heinrichstra\ss{}e 36, University of Graz, 8010 Graz, Austria 
    (\email{kristian.bredies@uni-graz.at}, \email{richard.huber@uni-graz.at}).
  NAWI Graz \url{https://www.nawigraz.at/}, BioTechMed Graz \url{https://www.biotechmedgraz.at/}.}
}

\usepackage{amsopn}

\DeclareMathOperator{\linspan}{span}
\DeclareMathOperator{\supp}{supp}

\usepackage{amsfonts}
\newcommand{\RR}{\mathbb{R}}

\newcommand{\nchoice}{n}

\newcommand{\Radon}{\mathcal{R}}
\newcommand{\RadonHat} {\overline{\Radon}_{\delta_s}}
\newcommand{\RadonS} {\Radon_{\delta_s}}
\newcommand{\RadonPhi}{\Radon_{\delta_s,\delta_\varphi}}
\newcommand{\RadonApp}{\Radon_{\delta_s,\delta_\varphi}^{\delta_x}}
\newcommand{\RadonDisc}{\mathbf{R}}

\newcommand{\RadonLA}{\Radon^{\AL}}

\newcommand{\RadonSA}{\Radon^{\mathcal{F}}}

\newcommand{\Fanbeam}{\mathcal{F}}
\newcommand{\FanbeamHat}{\overline{\Fanbeam}_{\delta_\xi}}

\newcommand{\FanbeamApp}{\Fanbeam_{\delta_\xi,\delta_\alpha}^{\delta_x}}
\newcommand{\FanbeamDisc}{\mathbf{F}}

\newcommand{\myG}{\mathcal{G}}
\newcommand{\myGApp}{\myG_{\delta_\xi,\delta_\alpha}^{\delta_x}}
\newcommand{\myGS}{\myG_{\delta_\xi}}
\newcommand{\myGPhi}{\myG_{\delta_\xi,\delta_\alpha}}

\newcommand{\hausdorff}[1]{\mathcal{H}^{#1}}
\newcommand{\restrict}{\,\llcorner\,}

\newcommand{\DW}{W}
\newcommand{\RE}{R_E}
\newcommand{\RD}{R}

\newcommand{\AL}{\mathcal{A}}

\newcommand{\mM}{\mathcal{M}}

\newcommand{\modcont}[1]{\omega_{#1}}
\newcommand{\dd}[1]{\,\mathrm{d}{#1}}
\newcommand{\inprod}{\cdot}
\newcommand{\set}[2]{\{{#1}\,:\,{#2}\}}
\newcommand{\bigset}[2]{\bigl\{{#1}\,:\,{#2}\bigr\}}
\newcommand{\sett}[1]{\{{#1}\}}
\newcommand{\abs}[1]{|{#1}|}
\newcommand{\bigabs}[1]{\bigl|{#1}\bigr|}
\newcommand{\Bigabs}[1]{\Bigl|{#1}\Bigr|}
\newcommand{\norm}[1]{\|{#1}\|}
\newcommand{\scp}[3][]{\langle{#2},{#3}\rangle_{#1}}
\newcommand{\conv}{\ast}
\newcommand{\placeholder}{\,\cdot\,}

\newcommand{\kronO}{\mathcal{O}}

\newcommand{\eulerE}{\mathrm{e}}

\ifpdf
\hypersetup{
  pdftitle={\titlename},
  pdfauthor={Kristian Bredies and Richard Huber}
}
\fi

\begin{document}

\maketitle

\begin{abstract}
This paper presents a novel mathematical framework for understanding pixel-driven approaches for the parallel beam Radon transform as well as for the fanbeam transform, showing that with the correct discretization strategy, convergence --- including rates --- in the $L^2$ operator norm can be obtained. These rates inform about suitable strategies for discretization of the occurring domains/variables, and are first established for the Radon transform. In particular, discretizing the detector in the same magnitude as the image pixels (which is standard practice) might not be ideal and in fact, asymptotically smaller pixels than detectors lead to convergence.  
Possible adjustments to limited-angle and sparse-angle Radon transforms are discussed, and similar convergence results are shown. In the same vein, convergence results are readily extended to a novel pixel-driven approach to the fanbeam transform. Numerical aspects of the discretization scheme are discussed,
and it is shown in particular that with the correct discretization strategy, the typical high-frequency artifacts can be avoided. 
\end{abstract}



\section{Introduction}
Projection-based tomography is a key tool for imaging in various scientific fields --- including medicine \cite{Hsieh_CT_principles}, materials science \cite{leary_analytical_2016}, astro-physics \cite{doi:10.1002/asna.200310234} and seismography  \cite{RAWLINSON2010101} --- as it allows to extract three-dimensional information from a series of two-dimensional projections.  Mathematically speaking, such tomography problems correspond to the inversion of the Radon transform \cite{Radon17,Ammari08_book_methematics_of_medical_imaging,Deans_Radon_applications_1993,Natterer:2001:MCT:500773}. That is, the line integral operator according to
\begin{equation}
\Radon f (s,\varphi)= \int_\mathbb{R} f( s \vartheta(\varphi)+t \vartheta^\perp(\varphi)) \dd{t},
\end{equation} i.e., the integral of a function $f$ along the line with projection angle $\varphi$, the associated normal and tangential vectors $\vartheta, \vartheta^\perp$, and detector offset $s$. Due to the high relevance of such imaging methods, many reconstruction approaches have been proposed, relevant examples include the filtered backprojection inversion formulas \cite{Radon17,Ammari08_book_methematics_of_medical_imaging}, iterative algebraic methods (e.g., ART, SART, SIRT) \cite{1951AmJM...73..615L,Gilbert_Sirt,SART_ALgo,GORDON1970471}, or variational imaging approaches \cite{Scherzer:2008:VMI:1502016,C8NR09058K,Dong2013,LaRoque:08,doi:10.1118/1.3371691}. 
Since all methods require some form of discrete version of the Radon transform and its adjoint --- the backprojection ---  a number of possible discretization schemes for the Radon transform  were proposed.

In this context, the class of ``fast schemes'' \cite{Averbuch_2D_Discrete_Radon,DRT_Beylkin_1987,Averbuch01fastslant,FDRT_Kelley_1993,1000260,doi:10.1137/S0097539793256673,KINGSTON20062040} consists of  approaches  which exploit connections between the Radon transform and the Fourier transform \cite{markoe_2006}. The algorithms are very efficient since they use the fast Fourier transform \cite{trove.nla.gov.au/work/21788406} and feature an ``explicit'' inversion formula,  allowing for  direct reconstruction.   This connection to the Fourier transform can, however, only be exploited under specific geometrical circumstances, making them unsuitable for most tomography applications \cite{Man_2004}.

Further, direct inversion schemes cannot always be applied. For instance,
in X-ray tomography, in order to reduce the radiation dose the sample or patient needs to endure, the number of measured projections is often reduced which makes the direct inversion unsuitable due to instability. To maintain the required quality of reconstructions, the use  of variational imaging methods became more prevalent,  in order to exploit prior information or assumptions \cite{1614066,Mairal:2014:SMI:2747300.2747301}. These methods do not require an exact inversion formula as they  consider an augmented or constrained inversion problem. Instead, a good, efficient and  widely applicable approximation of the Radon transform is needed.
 
To this point, distance-driven methods \cite{1239600,Man_2004,Chen_2015} and ray-driven methods \cite{Siddon1985FastCO,doi:10.1118/1.4761867,Path_through_pixels,Hsieh_CT_principles} were developed which are more flexible in comparison to Fourier methods. In the following, we only shortly discuss ray-driven methods, but similar observations can be made for distance-driven methods. Ray-driven methods consist of computing the line integral by discretizing the line itself and employing suitable %
quadrature formulas. %
A special case of this method consists of determining the length of the intersection of the line with any pixel and using these as weights in a sum over pixel values (which corresponds to using zero-order quadrature on the intersections).
 Note, however, that the determination of these weights is non-trivial and cannot easily be extended to higher dimensions. Moreover, the corresponding backprojection operators, i.e., the adjoints, generate strong artifacts, such that more straightforward discretizations of the adjoint are often used in practice, see, e.g.,~\cite{Xie_CUDA_paralelization_2015,Du2017,Syben_Python_reconstruction_in_NN_2019}.  Since ray-driven methods are efficient and versatile, they are prevalent in countless applications.

However, for the use of iterative methods such as in  Landweber-type approaches (e.g.,~SIRT) or in optimization steps of variational methods,  a proper backprojection is of great importance. Equally important, for these algorithms to work, it is (theoretically) necessary that the discrete Radon transform and discrete backprojection  are adjoint. Though widely used, ray-driven methods might not be ideal in this regard, as their adjoints
tend to introduce Moir\'{e} pattern artifacts, see e.g.~\cite{Liu_GPU_DDP_2017,Man_2004}. Thus, it might be reasonable to consider a projection method whose adjoint is a proper approximation of the backprojection in its own right.

To this point, one considers pixel-driven methods (in higher dimensions also voxel-driven methods) \cite{Hsieh_CT_principles,322963,Qiao2017ThreeNA,4331812}. These methods are based on a discretization of the backprojection via one-dimensional linear interpolation in the offset variable. This leads to a widely applicable Radon transform performing so-called ``anterpolation'' operations, which are the adjoints of interpolation. In this context, anterpolation means that pixels are projected onto the detector line, and the energy is linearly distributed onto the closest detectors with respect to the orthogonal distance. These methods admit a simpler structure than the ray-driven methods since instead of  taking the isotropic pixel structure into account, only the normal distance to lines is required. 
It is obvious from the derivation that the pixel-driven discretizations are adjoint and the backprojection is approximated reasonably well, but conversely, it is not obvious that the Radon transform is. This issue manifests in the fact that pixel-driven methods create strong oscillatory behavior (high-frequency artifacts) along some projection angles \cite{701833,1239600}, and therefore have gained little attention in practical applications in spite of its easy and efficient implementation and exact adjointness.

While the classical Radon transform considers parallel beams, some applications require different geometries, in particular, fanbeam or conebeam geometries \cite{Emissiontomo_2005,Natterer:2001:MCT:500773}. To reconstruct fanbeam data, rebinning --- recasting the data in a parallel setting at the cost of interpolation errors ---  can be used which then allows an inversion via the well-understood approaches for parallel CT \cite{DREIKE1976459}. For more sophisticated imaging methods, discretizations of the fanbeam transform and backprojection are required. To this point, many methods can be extended from the parallel beam to the fanbeam setting, see \cite{HaEfficientAR,8013154,Man_2004} and references therein. In particular, the same holds true for the pixel-driven approach \cite{HERMAN1976259,Johnson_1996,Zeng_Gullman1996}, though to the best of our knowledge, only the pixel-driven backprojection was considered for fanbeam geometry, but not the corresponding forward operator.  To this point, we propose a novel pixel-driven fanbeam transform which is adjoint to the pixel-driven backprojection and a proper discretization in its own right.

In the existing literature, there is only little discussion (see, e.g., \cite{Qiao2017ThreeNA,Man_2004,701833,870265}) of the worst-case error all these methods generate compared to the (true) continuous Radon transform or fanbeam transform and of what this error depends on. To the authors' best knowledge, there is no rigorous mathematical discussion on convergence properties for pixel-driven and ray-driven methods and in particular, no mathematical ``superiority'' of ray-driven or distance-driven methods was shown.
This paper aims at filling this gap and %
presents a rigorous convergence analysis of pixel-driven methods in a framework that easily allows the extension to pixel-driven methods for more general projection problems. This analysis shows that convergence, including rates, in the operator norm can be obtained if a suitable discretization strategy is pursued. In particular, this strategy leads to a suppression of high-frequency artifacts and thus informs that the reason for the oscillations being observed in the literature is not a defect of the method itself, but rather a consequence of unsuitable discretization parameters.

 The paper is organized as follows: Our main results are shown in  \Cref{Sec:Radontransform}  and consist in the mathematical framework and analysis of a pixel-driven parallel Radon transform discretization. After setting up the notation and definition in \Cref{Subsec:Radon_Def}, in \Cref{Subsec:Radon_Convergence}, convergence in operator norm to the continuous Radon transform is proven. In \Cref{Subsec:Limited_information}, adjustments to limitations in the angular range are considered, namely limited angles and sparse angles settings. In \Cref{Sec:Fanbeam}, the mathematical analysis is extended to the novel discrete fanbeam transform based on pixel-driven methods following a similar structure as \Cref{Sec:Radontransform}.
\Cref{Sec:Numeric} considers  numerical aspects of these discretizations, and discusses numerical experiments showcasing the practical applicability of the results. \Cref{sec:conclusions} concludes with some remarks and a brief outlook.

\section{The discrete Radon transform} \label{Sec:Radontransform}

\subsection{Derivation of pixel-driven methods} \label{Subsec:Radon_Def}
\usetikzlibrary{decorations.pathreplacing}
\usetikzlibrary{positioning,patterns}
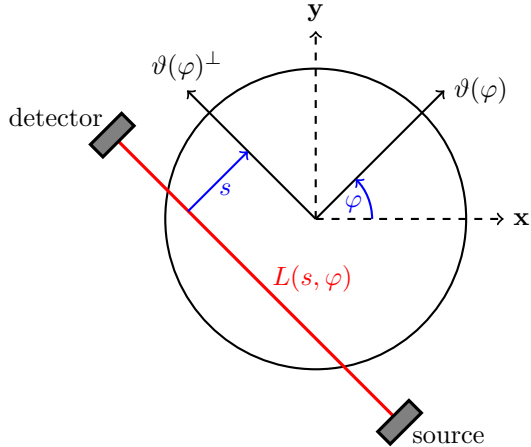
\begin{figure}[]
\center
 \begin{tikzpicture}
 
\draw[->,thick,dashed] (0,0) -- (0,2.5) node [above] {$\bold y$};
\draw[->,thick,dashed] (0,0) -- (2.5,0) node [right] {$\bold x$};

  \draw[thick] (0,0) circle (2cm );
\draw[->,thick] (0,0) -- (-1.41-0.3,+1.41+0.3)node[above] {$\vartheta(\varphi)^\perp$} ;

\draw[very thick,red](-.705-1.5-0.5,2.115-1.5+0.5)-- node[xshift=0.75cm,red]{$L(s,\varphi)$}  (2.115-1.5+0.5,-.705-1.5-0.5) ;
\draw[very thick] (-.705-1.5-0.5,2.115-1.5+0.5) (2.115-1.5+0.5,-.705-1.5-0.5) node (first)[draw, rotate=45,fill=gray] {\quad \quad} node [ right of =first,below of=first,node distance =0.3cm, rotate=0,xshift=0.35cm,yshift=0.1cm]{source};
\draw[very thick] (2.115-1.5+0.5,-.705-1.5-0.5) (-.705-1.5-0.5,2.115-1.5+0.5) node (first)[draw, rotate=45,fill=gray] {\quad \quad} node [ left of =first,above of =first, node distance =0.6cm,rotate=0,yshift=-0.35cm,xshift=-0.15cm]{detector};

\draw[thick,->] (0,0) -- (1.41+0.3,1.41+0.3) node[above,right] {$\vartheta(\varphi)$}  ;

\draw[thick,<-,blue] (-.9,.9) -- node [blue,midway,xshift=0.1cm,yshift=-0.1cm] {$s$} (-0.9-0.8,.9-0.8);

\draw[->,thick,blue] (0.75,0) arc (0:45:0.75);
  \node[blue,xshift=0.1cm,yshift=-0.05cm] at (0.4,0.25)  {$\varphi$};

  \end{tikzpicture}
  \caption{Geometry for the Radon transform. Source, detector and the connecting line $L(s,\varphi)$ parametrized by $t \mapsto s \vartheta(\varphi) + t   \vartheta(\varphi)^\perp$, where $\vartheta(\varphi)$
    is the projection direction and $s$ the detector offset. The direction $\vartheta(\varphi)^\perp$ corresponds to a rotation of $\vartheta(\varphi)$ by $\frac{\pi}2$ and is parallel to $L(s,\varphi)$.}
    \label{Fig:radon_geometry}
\end{figure}

In this subsection we motivate the pixel-driven approach by approximation of the continuous Radon transform in multiple steps, thus allowing to interpret it from a rigorous mathematical perspective. 
Moreover, we describe the framework and set up the notation used in this section.

Let $\Omega = B(0,1)$ be the 2-dimensional unit ball and
$\Omega' = {]{-1,1}[} \times S^1$, with all functions defined on
$\Omega$ and $\Omega'$ being extended by zero to $\RR^2$ and
$\RR \times S^1$, respectively. We will tacitly identify $[-\pi,\pi[$
with $S^1$ via the transformation
$\vartheta(\varphi)= \left(\cos(\varphi), \sin(\varphi)\right)$ such
that $\Omega'$ is identified with $\mathbb R \times [-\pi,\pi[$.

\begin{definition}
The \emph{Radon transform} of a compactly supported continuous function
$f \colon \RR^2 \to \RR$ 
is defined as %
\begin{equation} \label{equ:def_continuous_Radontrans}
  [\Radon
  f](s,\varphi) = \int_{\RR} f\bigl(s \vartheta(\varphi) + t
  \vartheta(\varphi)^\perp \bigr) \dd{t} = \int_{\set{x \in \RR^2}{x \inprod
      \vartheta(\varphi) = s}} f(x) \dd{\hausdorff{1}(x)}
\end{equation}
for $(s, \varphi) \in \RR \times {[{-\pi,\pi}[}$, where 
$\vartheta(\varphi)^\perp = \left(-\sin(\varphi), \cos(\varphi)\right)$ and
$\hausdorff{1}$ denotes the one-dimensional Hausdorff measure
\cite{folland1984real}.
The \emph{backprojection}
for $g: \RR \times S^1 \to \RR$ continuous and compactly supported
is given by
\begin{equation}
  \label{eq:def_backprojection}
  [\Radon^* g](x) = \int_{[-\pi,\pi[} g(x\cdot \vartheta(\varphi),\varphi)\dd\varphi
  \qquad\text{for} \ x \in \RR^2.
\end{equation}
\end{definition}
See~\cref{Fig:radon_geometry} for an illustration of the underlying geometry.
Considering $f$ supported on $\Omega$, definition
\cref{equ:def_continuous_Radontrans} can extended to a linear and
continuous operator $\Radon: L^2(\Omega) \to L^2(\Omega')$. Likewise,
$\Radon^*$ according to~\cref{eq:def_backprojection} yields a linear
and continuous operator $L^2(\Omega') \to L^2(\Omega)$. These
operators are indeed adjoint. %
The backprojection is often required in the context of tomographic
reconstruction
methods %
where both Radon transform and backprojection need to be discretized
in practice.  %
In order to justify the use of these operators in iterative
reconstruction methods, it is important for the discrete Radon
transform and the discrete backprojection to be adjoint operations.
However, adjointness of the discrete operations does not automatically
follow if the operators are discretized independently, which is a
common strategy in applications.

In the following, we derive the pixel-driven approach from a
mathematical perspective, allowing for an interpretation in terms of
approximation properties. The approach bases on approximating the line
integral in \eqref{equ:def_continuous_Radontrans} by an area integral
via
\begin{align}
  [\RadonHat f](s, \varphi) &= \frac1{\delta_s^2} \int_{\RR^2} w_{\delta_s}(x \inprod \vartheta(\varphi) - s) f(x) \dd{x}
  \\&= \frac{1}{\delta_s^2}\int_{\mathbb{R}} w_{\delta_s}(t-s) [\Radon f](t,\varphi) \dd t= \Bigl[ [\Radon f](\placeholder,\varphi) \conv \frac{w_{\delta_s}}{\delta_s^2} \Bigr] (s), \notag
\end{align}
where $w_{\delta_s}(t) = \max(0, \delta_s - \abs{t})$ and $\delta_s>0$ is an approximation parameter. Since the Radon transform corresponds, for each angle, to the convolution with a line measure, an approximation is found by the convolution with a hat-shaped function with width $2 \delta_s$. From a modeling perspective, this can be understood as accounting for detectors of the size $\delta_s$ possessing hat-shaped ``sensitivity profiles''.  The corresponding adjoint of the approximation is itself a reasonable approximation of the backprojection,
which can be described as
\begin{align}
  \notag
  [(\RadonHat)^* g] (x) &=
  \frac{1}{\delta_s^2}\int_{[-\pi,\pi[} \int_\RR w_{\delta_s}(x\cdot \vartheta(\varphi)-s) g(s,\varphi) \dd{s} \dd{\varphi} %
      = \Bigl[\Radon^*  (g \conv_1 \frac{w_{\delta_s}}{\delta_s^2} )\Bigr] (x),
\end{align}
where
$\conv_1$ denotes the convolution along the offset direction
$s$.  %
In the discrete Radon transform and backprojection that we derive in
the following, the local averaging after transformation becomes an
anterpolation step while the local averaging before the backprojection
becomes an interpolation step.

Next, we aim at discretizing these integrals on suitable discrete image and
sinogram spaces. First, we choose the discrete sinogram space
associated with a set of $Q$ angles
$\varphi_1,\ldots,\varphi_Q \in {[{-\pi,\pi}[}$,
$\varphi_1 < \varphi_2 < \ldots < \varphi_Q$, and an equispaced grid
of $P$ offsets $s_1,\ldots,s_P \in \RR$ such that
$s_p = \delta_s \bigl(p - \frac{P+1}2 \bigr)$ for each $p$ and some detector width
$\delta_s > 0$ (typically, $\delta_s = 2/P$).  A sinogram pixel is the product $S_p \times \Phi_q$
where $S_p = s_p + {[{-\delta_s/2, \delta_s/2}[}$ and
$\Phi_q = {[{(\varphi_{q-1} + \varphi_q)/2, (\varphi_{q} +
    \varphi_{q+1})/2}[}$ where $\varphi_0 = \varphi_Q - 2\pi$,
$\varphi_{Q+1} = \varphi_1 + 2\pi$ and the intervals are taken modulo
$2\pi$.  We also denote by
$\delta_\varphi = \max_{q=1,\ldots,Q} \varphi_{q+1} - \varphi_q$ the
angular discretization width.
The image is discretized by a $N \times M$ grid with pixel
size $\delta_x > 0$ and grid points
$x_{ij} = \delta_x \bigl(i - (N+1)/2, j - (M+1)/2\bigr)$. The associated pixel is then
$X_{ij} = x_{ij} + {[{-\delta_x/2, \delta_x/2}[}^2$,
the associated discrete spaces are given by
\begin{equation}\label{equ:def_discrete_image_sinogram_spaces}
  \begin{aligned}
  U &= \linspan\set{\chi_{X_{ij}}}{i = 1,\ldots,N, \
    j = 1,\ldots,M}, %
  \\
  \qquad
  V &= \linspan\set{\chi_{S_p \times \Phi_q}}{p=1,\ldots,P,
    \ q=1,\ldots,Q}, 
  \end{aligned}
\end{equation}
equipped with the scalar products on $L^2(\RR^2)$ and
$L^2(\RR \times S^1)$, respectively.  They can be identified with
$U = \RR^{N \times M}$ and $V = \RR^{P \times Q}$ equipped with the
scalar products
\[
  \scp[U]{f}{u} = \delta_x^2 \sum_{i,j=1}^{N,M} f_{ij}u_{ij}
  \quad \text{and} \quad
  \scp[V]{g}{v} = \delta_s \sum_{p=1}^P \sum_{q=1}^Q \Delta_q g_{pq}
  v_{pq},
\]
where $\Delta_q= (\varphi_{q+1}-\varphi_{q-1})/2$ denotes the
length of $\Phi_q$.

Provided that the
support of $f$ is contained in the union of all pixels, we can
discretize $f$ by
\begin{equation}
  f_{\delta_x} = \sum_{i,j=1}^{N,M} \delta_x^2 f_{ij} \delta_{x_{ij}} ,
  \quad
  f_{ij} = \frac{1}{\delta_x^2} \int_{X_{ij}} f(x) \dd{x}, \quad \text{for } i=1,\dots, N, \ j=1,\dots, M,\label{eq:discretization_image}
\end{equation}
where $\delta_{x_{ij}}$ corresponds to a delta peak in $x_{ij}$, i.e.,
one replaces $f$ by delta peaks in the pixel centers weighted by their
area $\delta_x^2$. Note that $\delta_x^2 f_{ij}$ corresponds to
the total mass associated with the pixel $X_{ij}$, i.e., the mass of
each pixel is shifted into its center.
 
The approximation $\RadonHat$ can still be applied to $f_{\delta_x}$, leading
to the semi-discrete Radon transform
\begin{equation}
[\RadonHat f_{\delta_x}] (s, \varphi) = \frac{\delta_x^2}{\delta_s^2} 
\sum_{i,j=1}^{N,M} w_{\delta_s}(x_{ij} \inprod \vartheta(\varphi) - s) f_{ij}.
\end{equation}
Further restricting to functions %
that are piecewise constant on the partition $(S_p\times \Phi_q)_{pq}$ with values extrapolated from the values in $(s_p,\varphi_q)$  yields the following definition.
\begin{definition} 
  The \emph{fully discrete Radon transform} is defined by
\begin{equation} \label{equ:def_Radon_all_discrete}
 [\RadonApp f](s, \varphi) = \frac{\delta_x^2}{\delta_s^2} \sum_{p=1}^P \sum_{q=1}^Q \chi_{S_p}(s)\chi_{\Phi_q}(\varphi) 
\sum_{i,j=1}^{N,M} w_{\delta_s}(x_{ij} \inprod \vartheta(\varphi_q) - s_p) f_{ij}.
\end{equation}
The corresponding mapping between the pixel spaces $U$,
$V$ and their identification in terms of pixel values is denoted by
\begin{equation} \label{equ:def_discrete_radon} 
  \RadonDisc \colon U \to V, \quad \quad  [\RadonDisc f]_{pq} =
  \frac{\delta_x^2}{\delta_s^2} \sum_{i,j=1}^{N,M} w_{\delta_s}(x_{ij} \inprod
  \vartheta(\varphi_q) - s_p) f_{ij}.
\end{equation}
\end{definition}
 The operator
 $\RadonDisc$ distributes,
for each $q$,
 the intensity $f_{ij}$ of each pixel
$X_{ij}$ to the
$p$-th detector according to the weights $w_{\delta_s}(x_{ij} \inprod
\vartheta(\varphi_q) -
s_p)$. This is the anterpolation operation that appears in the context
of pixel-driven Radon transforms. For fixed
$(i,j)$, there are at most two
$p$ for which the weight $w_{\delta_s}(x_{ij} \inprod
\vartheta(\varphi_q) -
s_p)$ is non-zero.  Summarized, the pixel-driven approach has three
ingredients: The approximation of line measures by hat-shaped
functions, the discretization of images by lumping the mass of pixels
to their centers and the extrapolation of sinogram pixels from the
values at their centers.

The adjoint of the fully discrete Radon transform reads as
\begin{equation}\label{equ:def_approx_backprojection}
  [(\RadonApp)^*g](x)= \sum_{i,j=1}^{N,M} \chi_{X_{ij}}(x) \sum_{p=1}^P \sum_{q=1}^Q
  \frac{\Delta_q}{\delta_s} w_{\delta_s}(x_{ij}\cdot \vartheta_q-s_p)
  g_{pq},
\end{equation}
where $g_{pq} = \frac{1}{\delta_s \Delta_q}\int_{S_p}\int_{\Phi_q}
g(s,\varphi) \dd{\varphi}\dd{s}$ and $x \in
\Omega$. On the discrete spaces $U$ and $V$, this means
\begin{equation} \label{equ:def_discrete_backprojection}
  \RadonDisc^*\colon V \to U, \quad \quad  [\RadonDisc^* g]_{ij}=
  \sum_{q=1}^Q \Delta_q \sum_{p=1}^P \frac{1}{\delta_s}
  w_{\delta_s}(x_{ij}\cdot \vartheta(\varphi_q)-s_p)g_{pq}.
\end{equation}
Here, the sum over
$p$ contains at most two non-zero elements. Except on the detector
boundary, $p$ can uniquely be chosen such that $s_{p} < x_{ij} \inprod
\vartheta(\varphi_q) \leq s_{p+1}$, leading to only $p$ and
$p+1$ contributing to the sum. By definition, the latter is then the
linear interpolation of $g_{pq}$ and $g_{(p+1)q}$ at $s_p$ and $s_{p +
  1}$ to the detector offset $x_{ij} \inprod
\vartheta(\varphi_q)$, yielding the well-known form of the
pixel-driven backprojection.

In summary, pixel-driven methods can be considered the result of an abstract
approximation and a subsequent step-by-step discretization of the occurring
variables, such that in each step, the abstract understanding is maintained.
 This allows for a clearer mathematical interpretation and motivates the theoretical procedure in the following section.

\subsection{Convergence analysis} \label{Subsec:Radon_Convergence}
Following the motivation in the previous section, we consider the error of switching from line to area integral as well as the discretization of the occurring functions in order to obtain convergence results.

We identify $\varphi \in {[{-\pi, \pi}[}$ with
$\vartheta(\varphi) \in S^1$ and let
$\dd{\vartheta} = \dd{(\hausdorff{1} \restrict S^1)}$ as well as
$\Theta_q = \set{\vartheta(\varphi)}{\varphi \in \Phi_q}$.  In
particular, we treat $S^1$ as an additive group which realizes
addition modulo $2\pi$ and denote by $\abs{\vartheta}$ the smallest
non-negative $\varphi$ such that $\vartheta(\varphi) = \vartheta$.
 Further, in the
following, the discretization is always assumed to be compatible with
$\Omega$ and $\Omega'$, i.e., $\Omega$ is contained in the union of
all image pixels $X_{ij}$ and $\Omega'$ is contained in the union of
all sinogram pixels $S_p \times \Theta_q$. All operator norms we consider
in the following relate to operators $L^2(\Omega) \to L^2(\Omega')$.

\begin{definition}
The $L^2$ modulus of continuity of a function $g \in L^2(\RR \times S^1)$ is 

\[
\modcont{g}(h, \gamma) = \Bigl( \int_{S^1} \int_{\RR} \abs{g(s+h, \vartheta
  + \gamma) - g(s,\vartheta)}^2 \dd{s} \dd{\vartheta} \Bigr)^{1/2}.
\]
\end{definition}
The asymptotic behavior %
for vanishing $h$ and $\gamma$ is a measure of regularity: For
instance, for $g \in L^2(\Omega')$, we have that
$g \in H^1_0(\Omega')$ if and only if
$\modcont g(h,\gamma) = \kronO(\abs{h} + \abs{\gamma})$, and
$g \in H^\alpha_0(\Omega')$, $0 < \alpha < 1$, if  $\int_{S^1} \int_\RR (\abs{h}^2 + \abs{\gamma}^2)^{-(\alpha+1)}\modcont g(h,\gamma)^2 \dd{h} \dd{\gamma} < \infty$ (see \cite{10.2307/j.ctt1bpmb07}).

We are interested in the asymptotic behavior of the modulus of continuity for $g = \Radon f$ and $\gamma = 0$ in order to show that the Radon transformation generates  regularity in the offset dimension. 

\begin{lemma}\label{lem:radon_mod_continuity}
  Let $f \in L^2(\Omega)$ and $g = \Radon f$.  Then,
  $\modcont{g}(h, 0) \leq c \sqrt{\abs{h}} \norm{f} $ for every
$h \in \RR$  and some constant $c > 0$ independent of $f$ and  $h$.
\end{lemma}

\begin{proof}
  Denote by $T_h$ the translation operator associated
  with $(h,0)$, i.e., for $g \in L^2(\RR \times S^1)$, we have
  $[T_h g](s,\vartheta) = g(s+h, \vartheta)$.
  Then,
  $\norm{T_h g} = \norm{g}$ implying
  $\norm{T_h g - g}^2 = 2\scp{g - T_hg}{g}$ and plugging in
  $g = \Radon f$ gives
  \[
    \modcont g(h,0)^2 \leq 2\norm{\Radon^*\Radon f -
      \Radon^* T_h \Radon f} \norm{f}
  \]
  by virtue of the Cauchy--Schwarz inequality.
  We compute, for $\tilde f \in L^2(\Omega)$ that
  \[
  \begin{aligned}
    \scp{T_h\Radon f}{\Radon\tilde f} &= \int_{S^1}
    \int_{\RR} \int_{\RR} \int_{\RR} f(\vartheta (s + h) +
    \vartheta^\perp t) \tilde{f}(\vartheta s + \vartheta^\perp \tau)
    \dd{t} \dd{\tau} \dd{s}
    \dd{\vartheta} \\
    &= \int_{\Omega} \Bigl( \int_{S^1} \int_{\RR} f(x + \vartheta h +
    \vartheta^\perp t-(x\cdot\vartheta^\perp)  \vartheta^\perp)
    \dd{t} \dd{\vartheta} \Bigr) \tilde f(x) \dd{x} \\
    &= 2 \int_\Omega \Bigl( \int_{\abs{x - y} \geq \abs{h}}
    \frac{1}{\sqrt{\abs{x - y}^2 - h^2}} f(y) \dd{y} \Bigr)
    \tilde f(x) \dd{x},
  \end{aligned}
  \]
  where we substituted $x = \vartheta s + \vartheta^\perp \tau$ for
  $(s, \tau)$ and $y = x + \vartheta h + \vartheta^\perp t-(x\cdot\vartheta^\perp ) \vartheta^\perp$ for
  $(t, \vartheta)$. Denoting by
  \[
  k_h(x,y) =
  \begin{cases}
    0 & \text{if} \ \abs{x - y} < \abs{h},\\
    \frac{1}{\sqrt{\abs{y - x}^2 - h^2}} & \text{if} \ \abs{x -
      y} \geq \abs{h},
  \end{cases}
  \]
  we have that
  $[\Radon^*T_h\Radon f](x) = 2 \int_\Omega k_h(x,y) f(y)
  \dd{y}$, i.e., the operator corresponds to a convolution.
  Due to Young's inequality,   with  
  \begin{multline} M_h = \sup_{x \in \Omega} \int_{\Omega} \abs{k_0(x,y) - k_h(x,y)}
  \dd{y}= \sup_{x\in \Omega} \int_{x-\Omega} \abs{k_0(0,y) - k_h(0,y)}
  \dd{y}
	\\ \leq \int_{|y|\leq 2} \abs{k_0(0,y) - k_h(0,y)}
  \dd{y} 
  ,
  \end{multline}
  we can estimate
  $\norm{(\Radon^*\Radon -\Radon^*T_h\Radon)f}
  \leq 2 M_h \norm{f}$.  For $|h| \leq 2$, $M_h$ can be estimated 
  by changing to polar coordinates as follows:
  \[
  \begin{aligned}
    M_h &\leq \int_{\abs{y} \leq \abs{h}} \frac1{\abs{y}} \dd{y} +
    \int_{\abs{h} \leq \abs{y} \leq 2} \frac{1}{\sqrt{\abs{y}^2 -
        h^2}} - \frac{1}{\abs{y}} \dd{y} \\
    & %
    =2\pi \abs{h} + 2\pi \int_{\abs{h}}^2 \frac{r}{\sqrt{r^2 - h^2}} - 1
    \dd{r} = 2\pi(2\abs{h}  \underbrace{-2 + \sqrt{4 - h^2}}_{\leq 0} %
    ) \leq 4 \pi \abs{h}.
  \end{aligned}
  \]
  If $\abs{h} > 2$, then
  $M_h = \int_{\abs{y} \leq 2} k_0(0,y) \dd{y} = 4\pi \leq
  4\pi\abs{h}$.  Together, we thus get
  $\modcont g(h,0)^2 \leq 16\pi \abs{h} \norm{f}^2$ which proves the
  claim.
\end{proof}

Next, denote by $\RadonS$ the operator
$\RadonHat$ that is additionally discretized with respect to the offset parameter $s$, i.e.,
\begin{equation}
  \label{eq:radon_detector_discretization}
[\RadonS f](s,\vartheta) = \frac1{\delta_s^2} \sum_{p=1}^P
\chi_{S_p}(s) \int_\RR 
w_{\delta_s}(t - s_p) \Radon f(t, \vartheta) \dd{t}.
\end{equation}
We are interested in the norm of the difference of
$\RadonS$ and $\Radon$, i.e., the error of approximating the line integral by the area integral and discretizing the offset.
\begin{lemma}
  \label{lem:radon_detector_discretization}
  For %
  $f \in L^2(\Omega)$, we have
  $\norm{\RadonS f - \Radon f}
  \leq C \sup_{\abs{h} < \frac32 \delta_s} \modcont{\Radon f}(h,0)$.
\end{lemma}

\begin{proof}
  For $f \in L^2(\Omega)$ and $(s,\vartheta) \in \Omega'$ we compute
  \[
  [\RadonS f](s,\vartheta) -
  [\Radon f](s,\vartheta) = \frac1{\delta_s^2}\sum_{p=1}^P
  \chi_{S_p}(s)
  \int_{\RR}
  w_{\delta_s}(t - s_p) \bigl( [\Radon f](t,\vartheta) -
  [\Radon f](s,\vartheta) \bigr) \dd{t}
  \]
  since
  $\frac1{\delta_s^2}\int_\RR w_{\delta_s}(t) \dd{t} = 1$, and with Jensen's   inequality we get
  \begin{multline*}
    \norm{\RadonS f -
      \Radon f}^2 \\
    \leq \frac1{\delta_s}
    \int_{S^1} \int_\RR \int_\RR \Bigl[ \sum_{p=1}^P 
    \chi_{S_p}(s)
    \frac{w_{\delta_s}(t - s_p)}{\delta_s} %
    \Bigr]
    \bigabs{[\Radon f](t,\vartheta) -
      [\Radon f](s,\vartheta)}^2 \dd{t} \dd{s} \dd{\vartheta}.
  \end{multline*}
  If $\abs{t - s} \geq \frac32 \delta_s$, then
  $s \in S_p$
  and $\abs{t - s_p} < \delta_s$ cannot
  hold at the same time, so these $(s,t)$ do not contribute to the
  integral on the right-hand side. If
  $\abs{t - s} < \frac32 \delta_s$, there is at most one $p$ for
  which %
  $s \in S_p$,
  such that the sum over $p$ can
  be estimated by $1$. Hence, substituting $h = t - s$
  leads %
  to the desired estimate:
  \begin{align}  \label{equ:Proof_lastequation_2.2}
    \norm{\RadonS f - \Radon f}^2 &\leq
    \frac1{\delta_s} \int_{S^1} \int_\RR \int_{\abs{t - s} < \frac32
      \delta_s} \bigabs{[\Radon f](t,\vartheta) -
      [\Radon f](s,\vartheta)}^2
    \dd{t} \dd{s} \dd{\vartheta}     
    \\    
                                  &= \frac1{\delta_s} \int_{\abs{h} < \frac32 \delta_s} %
                                    \int_{S^1}
    \int_{\RR} \bigabs{[\Radon f](s+h,\vartheta) -
      [\Radon f](s,\vartheta)}^2
                                    \dd{s} \dd{\vartheta}%
                                    \dd{h} \notag
    \\
    & = \frac1{\delta_s} \int_{\abs{h} < \frac32 \delta_s}  \modcont{\Radon f}(h,0)^2 \dd{h} \leq
      3 \sup_{\abs{h} < \frac32 \delta_s} \modcont {\Radon f}(h,0)^2.
      \notag
  \end{align}
\end{proof}
The previous lemma combined with \cref{lem:radon_mod_continuity} implies
that at least,
$\norm{\RadonS - \Radon} = \kronO(\delta_s^{1/2})$, but depending
on the regularity of $\Radon f$ in terms of the modulus of continuity,
also higher rates may be achieved for specific $f$. The following
lemma shows that the modulus of continuity can also be used to
estimate the approximation error between the adjoints of $\RadonS$ and
$\Radon$, respectively.

\begin{lemma} \label{lem:radon_detector_discretization_Backprojection}
The adjoint of $\RadonS$ is 
\begin{equation}
[(\RadonS)^*g](x)=\frac{1}{\delta_s^2}\sum_{p=1}^P \int_{S^1} w_{\delta_s}(x\cdot \vartheta -s_p) \int_{S_p}g(s,\vartheta) \dd{s} \dd{\vartheta},
\end{equation}
and the approximation error for the adjoint for $g \in L^2(\Omega')$
can be estimated by
\begin{equation}
\norm{(\RadonS)^* g - \Radon^* g}\leq C \sup_{\abs{h} < \frac32 \delta_s} \modcont g(h,0).
\end{equation}
\end{lemma}

\begin{proof}
The representation of the adjoint is readily computed. 
Inserting the definitions of the occurring operators  and putting the $g(x\cdot \vartheta,\vartheta)$ term into the inner-most sum and integral yields
\begin{multline*}
\norm{(\RadonS)^*g-\Radon^* g}^2 
\\
  = \int_{\Omega} \Bigabs{ \int_{S^1} \sum_{p=1}^P
    \frac{w_{\delta_s}(x \cdot \vartheta - s_p)}{\delta_s}
    \Bigl( \frac{1}{\delta_s}\int_{S_p} g(s,\vartheta)
    -g(x\cdot \vartheta, \vartheta) \dd{s} \Bigr) \dd{\vartheta} }^2 \dd{x},
\end{multline*}
where we exploited that for $x \in \Omega$, we have
$\frac{1}{\delta_s}\sum_{p=1}^P w_{\delta_s}(x \cdot
\vartheta(\varphi)-s_p)=1$ and $\frac{1}{\delta_s}\int_{S_p} 1 \dd{s}=1$.
Applying the Cauchy--Schwarz inequality as well as Jensen's
inequality, the fact that
$w_{\delta_s}(x \cdot \vartheta-s_p)\neq 0$ and $s\in S_p$
implies $|x\cdot \vartheta-s| < \frac{3}{2} \delta_s$, as
well as the change of variables $h = s - x\cdot \vartheta $
gives
\begin{align*}
\|(\RadonS)^*g - &\Radon^* g\|^2
\\
& \leq 2\pi \int_{\Omega}   \int_{S^1}\sum_{p=1}^P \frac{w_{\delta_s}(x \cdot \vartheta-s_p)}{\delta_s} \Bigabs{\frac{1}{\delta_s}\int_{S_p} g(s,\varphi)
 -g(x\cdot \vartheta, \vartheta) \dd{s}}^2 \dd{\vartheta}  \dd{x}
\\
&\leq 2\pi  \int_{\Omega}   \int_{S^1} \Bigabs{\frac{1}{\delta_s}\int_{|h| < \frac{3}{2}\delta_s} |g(x\cdot \vartheta + h,\vartheta)
 -g(x\cdot \vartheta, \vartheta) | \dd{h} }^2 \dd{\vartheta}  \dd{x}.
\end{align*}
Interchanging the order of integration, substituting
$x = s \inprod \vartheta + t \inprod \vartheta^\perp$, interchanging
integration order once again, and applying the Cauchy--Schwarz
inequality finally implies
\begin{align*}
\norm{(\RadonS)^*g-\Radon^* g}^2
&\leq 2\pi  \int_{{]{-1,1}[}} \int_{S^1} \int_{\mathbb{R}} \Bigabs{\frac{1}{\delta_s}\int_{|h|\leq \frac{3}{2}\delta_s} |g(s+h,\vartheta)
 -g(s,\varphi)| \dd{h} }^2 \dd{s} \dd{\vartheta} \dd{t}
\\&
\leq \frac{12\pi}{\delta_s}  \int_{|h| < \frac{3}{2}\delta_s}   \int_{S^1}  \int_{\mathbb{R}}   \abs{ g(s+h,\vartheta)
 -g(\tau,\vartheta) }^2  \dd{s} \dd{\vartheta} \dd{h} 
 \\
  &= \frac{12\pi}{\delta_s}  \int_{|h| < \frac{3}{2}\delta_s} \modcont{g}(h,0)^2 \dd{h} \leq  36\pi
    \sup_{|h| < \frac{3}{2}\delta_s}\modcont g(h,0) ^2.
\end{align*}
\end{proof}

Next, we estimate the difference between $\RadonS$ and
the operator that also discretizes the angle variable $\vartheta$:
\begin{equation}
    \label{eq:radon_detector_angle_discretization}
  [\RadonPhi f](s, \vartheta) = \frac1{\delta_s^2}
  \sum_{q = 1}^Q \sum_{p=1}^P \chi_{S_p}(s) \chi_{\Theta_q}(\vartheta)
  \int_\Omega w_{\delta_s}(x \inprod \vartheta_q - s_p) f(x) \dd{x}.
\end{equation}

\begin{lemma}
  \label{lem:radon_detector_angular_discretization}
  We have that
  $\norm{\RadonPhi - \RadonS} \leq C \frac {\delta_\varphi}{\delta_s}$.
\end{lemma}

\begin{proof}
  Let $f \in L^2(\Omega)$ and fix $p \in \sett{1, \ldots, P}$.
  Via the Cauchy--Schwarz inequality, we obtain
  \[
    \begin{aligned}
	\norm{ [\RadonPhi f-\RadonS f] (s_p,\cdot)}^2_{L^2(S^1)}   =   
      \int_{S^1} |[\RadonPhi f](s_p, \vartheta)
      - [\RadonS f](s_p, \vartheta)|^2 \dd{\vartheta} \\
      \leq \frac1{\delta_s^4} \int_{S^1} \sum_{q=1}^Q
      \chi_{\Theta_q}(\vartheta) \int_{\Omega} \abs{w_{\delta_s}(x
        \inprod \vartheta_q - s_p) - w_{\delta_s} (x \inprod \vartheta
        - s_p)} \dd{x} 
      \\
       \qquad \cdot
      \int_{\Omega} \abs{w_{\delta_s}(y \inprod
        \vartheta_q - s_p) - w_{\delta_s} (y \inprod \vartheta - s_p)}
      \abs{f(y)}^2 \dd{y} \dd{\vartheta}.%
    \end{aligned}
  \]
  Fix $q \in \sett{1,\ldots,Q}$, $\vartheta \in S^1$ and choose
  $\varphi$ as the smallest $\varphi \geq \varphi_q$ such that
  $\vartheta(\varphi) = \vartheta$. With
  $\xi(t) = \vartheta(\varphi_q + t)$ and denoting
  by $w'_{\delta_s}$ the weak derivative of $w_{\delta_s}$,
  we can estimate the integral with respect to $x$ as follows:
  \[
    \begin{aligned}
      \int_{\Omega} |w_{\delta_s}(x \inprod \vartheta_q - s_p) &-
      w_{\delta_s} (x \inprod \vartheta - s_p)| \dd{x}
      \\
      &\leq
      \int_{0}^{\abs{\vartheta - \vartheta_q}}
      \int_\Omega \abs{w_{\delta_s}'(x \inprod \xi(t)
        - s_p)} \abs{x \inprod 
      \xi(t)^\perp} \dd{x} \dd{t} \\
      & \leq 4 \delta_s \abs{\vartheta - \vartheta_q}
    \end{aligned}
  \]
  since for $\xi \in S^1$, the function
  $x \mapsto w_{\delta_s}'(x \inprod \xi - s)$ is supported on a
  stripe of width $2\delta_s$ within the unit ball $\Omega$. 
  Using that
  $\abs{w_{\delta_s}(x \inprod \vartheta_q - s_p) - w_{\delta_s}(x
    \inprod \vartheta - s_p)} \leq \abs{\vartheta - \vartheta_q}$,
  this leads to the $L^2(S^1)$-norm estimate
  \begin{multline*}
    \norm{[\RadonPhi f](s_p, \cdot) -
      [\RadonS f](s_p, \cdot)}^2_{L^2(S^1)} \\
    \leq \frac4{\delta_s^3} \Bigl(\sum_{q=1}^Q \int_{S^1} \chi_{\Theta_q}(\vartheta)
    \abs{\vartheta - \vartheta_q}^2 \dd{\vartheta} \Bigr) \int_\Omega
    \abs{f(y)}^2 \dd{y}.
  \end{multline*}
  Recalling the definition of $\Theta_q$ in terms of $\varphi_{q-1},\varphi_q$
  and $\varphi_{q+1}$, the sum with respect to $q$ can be estimated by
  \[
    \begin{aligned}
      \sum_{q=1}^Q \int_{S^1} \chi_{\Theta_q}(\vartheta) 
      \abs{\vartheta - \vartheta_q}^2 \dd{\vartheta}
      &= \frac1{24} \sum_{q=1}^Q (\varphi_{q} - \varphi_{q-1})^3 
      + (\varphi_{q+1} - \varphi_q)^3 \\
      & \leq \frac{\delta_\varphi^2}{24} \sum_{q=1}^Q (\varphi_q - \varphi_{q-1}) + (\varphi_{q+1} - \varphi_q) 
      = \frac{\pi}{6} (\delta_\varphi)^2.
    \end{aligned} 
  \]
  In total, we have
  $\norm{[\RadonPhi f](s_p, \cdot) -
   [ \RadonS f](s_p, \cdot)}^2 %
  \leq \frac{2\pi}{3}
  \delta_\varphi^2/\delta_s^3 \norm{f}^2$ which leads to the desired 
  $L^2(\Omega')$-estimate as follows:
  \[
    \norm{\RadonPhi f - \RadonS f}^2%
    = \delta_s \sum_{p=1}^P \norm{[\RadonPhi f](s_p,
    \cdot) - [\RadonS f](s_p, \cdot)}_{L^2(S^1)}^2 \leq
  \frac{2\pi}{3} \frac{\delta_\varphi^2}{\delta_s^2} \norm{f}^2.
  \]
\end{proof}
Finally, we replace $f$ by $f_{\delta_x}$, consider 
$\RadonApp f = \RadonPhi f_{\delta_x}$ which results in
\begin{equation}
  \label{eq:radon_detector_angle_space_discretization}
 [\RadonApp f] (s, \varphi) = \frac1{\delta_s^2}
  \sum_{q=1}^Q \sum_{p=1}^P \chi_{S_p}(s) \chi_{\Theta_q}(\vartheta)
  \sum_{i,j=1}^{N,M} w_{\delta_s}(x_{ij} \inprod \vartheta_q - s_p) \int_{X_{ij}}
  f(x) \dd{x}
\end{equation}
and compare it with $\RadonPhi f$. 

\begin{lemma}
  \label{lem:radon_image_discretization}
  It holds that $\norm{\RadonApp
    - \RadonPhi} \leq C \sqrt{1 + \frac {\delta_x}{\delta_s}} 
  \frac {\delta_x}{\delta_s}$.
\end{lemma}

\begin{proof}
  We proceed in analogy to the proof of
  \cref{lem:radon_detector_angular_discretization}.  Denote by
  $\Pi(x) = x_{ij}$ if $x \in X_{ij}$, i.e., the projection on the
  closest pixel center and observe that
  $\abs{\Pi(x) - x} \leq \frac1{\sqrt{2}} \delta_x$.  For
  $f \in L^2(\Omega)$, estimate
\begin{align*}
  &   \norm{ \RadonApp f - \RadonPhi f}^2%
    =
      \norm{\RadonPhi ( f_{\delta_x} -f )}^2 
      \\
      & \quad \leq
      \frac1{\delta_s^3} \sum_{q=1}^Q \sum_{p=1}^P 
        \Delta_q
        \Bigabs{\int_{\Omega} \bigl(w_{\delta_s}(\Pi(x)
        \inprod \vartheta_q - s_p) - w_{\delta_s}(x \inprod
        \vartheta_q - s_p) \bigr) f(x) \dd{x}}^2.
\end{align*}  %
  We %
  intend to use the Cauchy--Schwarz inequality on the integral with
  respect to $x$ and estimate further. For that purpose, observe that
  \begin{multline*}
    \int_{\Omega} \bigabs{w_{\delta_s}(\Pi(x) \inprod \vartheta_q -
      s_p) - w_{\delta_s}(x \inprod \vartheta_q - s_p)} \dd{x} \\
    \leq \int_0^1 \int_\Omega \bigabs{w_{\delta_s}'\bigl((x +
      t(\Pi(x) - x)) \inprod \vartheta_q - s_p \bigr)} \abs{\Pi(x) -
      x} \dd{x} \dd{t}.
  \end{multline*}
  Note that
  $\abs{x \inprod \vartheta_q - s_p} \geq \delta_x/\sqrt{2} +
  \delta_s$ implies
  $w_{\delta_s}'\bigl((x + t(\Pi(x) - x)) \inprod \vartheta_q - s_p
  \bigr) = 0$, hence
  \[
    \int_{\Omega} \bigabs{w_{\delta_s}(\Pi(x) \inprod \vartheta_q -
      s_p) - w_{\delta_s}(x \inprod \vartheta_q - s_p)} \dd{x} \leq 4
    \Bigl( \frac{\delta_x}{\sqrt{2}} + \delta_s \Bigr) \frac{\delta_x}{\sqrt{2}}.
  \]
  Also,
  \[
    \begin{aligned}
      \sum_{p=1}^P \bigl |w_{\delta_s}(\Pi(x) &\inprod \vartheta_q -
        s_p) - w_{\delta_s}(x \inprod \vartheta_q - s_p) \bigr| \\
        &\leq
      \int_0^1 \sum_{p=1}^P \bigabs{w_{\delta_s}'\bigl((x + t(\Pi(x) -
        x)) \inprod \vartheta_q - s_p \bigr)} \abs{\Pi(x) - x} \dd{t} 
      \leq 2 \frac{\delta_x}{\sqrt{2}},
    \end{aligned}
  \]
  since
  $\bigabs{w_{\delta_s}'\bigl((x + t(\Pi(x) - x)) \inprod \vartheta_q
    - s_p \bigr)}$ is $1$ for at most two $p$ and $0$ else. 
  Altogether, it follows for the $L^2(\Omega')$-norm  that
  
  \begin{equation*}
    \norm{\RadonPhi f - \RadonApp f}^2%
    \leq
    \frac4{\delta_s^3} \sum_{q=1}^Q
    \Delta_q \Bigl( \frac{\delta_x}{\sqrt{2}} + \delta_s
    \Bigr) \delta_x^2 \norm{f}^2
     \leq 8\pi \frac{\delta_x^2}{\delta_s^2} \Bigl(1 +
     \frac{\delta_x}{\delta_s}  \Bigr)
     \norm{f}^2,
     \end{equation*}
  which completes the proof.
\end{proof}

\begin{theorem}
  \label{thm:radon_convergence}
  If $\delta_s \to 0$, $\frac{\delta_\varphi}{\delta_s} \to 0$ and
  $\frac{\delta_x}{\delta_s} \to 0$, then
  $\RadonApp$ converges to $\Radon$ in
  operator norm for linear and continuous mappings
  $L^2(\Omega) \to L^2(\Omega')$.
  
  If additionally,
  $\delta_\varphi = \kronO(\delta_s^{1 + \epsilon})$ and
  $\delta_x = \kronO(\delta_s^{1 + \epsilon})$ for some
  $0 < \epsilon \leq \frac12$, then
  $\norm{\RadonApp - \Radon} =
  \kronO(\delta_s^\epsilon)$ as $\delta_s \to 0$.
\end{theorem}

\begin{proof}
  Combining \cref{lem:radon_mod_continuity} and
  \cref{lem:radon_detector_discretization} yields
  $\norm{\RadonS - \Radon} \leq C \sqrt{\delta_s}$, so together with
  Lemmas~\ref{lem:radon_detector_angular_discretization}
  and~\ref{lem:radon_image_discretization}, we get
  \[
    \begin{aligned}
      \norm{\RadonApp - \Radon} &\leq
      \norm{\RadonS - \Radon} +
      \norm{\RadonPhi - \RadonS}
      + \norm{\RadonApp - \RadonPhi} \\
      & \leq C \Bigl(\sqrt{\delta_s} + \frac{\delta_\varphi}{\delta_s}
      + \sqrt{1 + \frac{\delta_x}{\delta_s}} \frac{\delta_x}{\delta_s}
      \Bigr),
    \end{aligned}
  \]
  where the right-hand side vanishes if $\delta_s \to 0$, 
  $\frac {\delta_\varphi}{\delta_s} \to 0$ and $\frac{\delta_x}{\delta_s} \to 0$.
  
  If $\delta_\varphi = \kronO(\delta_s^{1 + \epsilon})$ and
  $\delta_x = \kronO(\delta_s^{1 + \epsilon})$ for some
  $0 < \epsilon \leq 1/2$, then in particular,
  $\sqrt{1 + \delta_x/\delta_s}$ stays bounded and $ \sqrt{\delta_s} = \kronO(\delta_s^\epsilon)$ as
  $\delta_s \to 0$, so the claimed rate follows.
\end{proof}

\begin{remark}
Note  that $\frac{\delta_x}{\delta_s} \to 0$ as $\delta_s \to 0$ is necessary for the upper bound of the discretization error in \cref{lem:radon_image_discretization} to vanish, suggesting that the standard choice $\delta_s \approx \delta_x$ might not be well suited and might in fact be the origin of oscillatory behavior described in literature.  This supports the observation in \cite{Qiao2017ThreeNA} that considering smaller image pixels and larger detectors can suppress high-frequency artifacts substantially.
However, it is important to note that this assumption concerning the discretization is sufficient for the stated convergence but we did not prove its necessity. Moreover, the standard setting with $\delta_x=\delta_s$ might still be suitable for weaker forms of convergence such as pointwise convergence.
\end{remark}

\begin{remark}
  Note that in view of inverting the Radon transform, which is ill-posed,
  the presented convergence result could become relevant when
  employing a ``regularization by discretization'' strategy, see, e.g.,
  \cite{H_marik_2016,Natterer_Projection}.
  Currently, this theory is not directly applicable to the type of discretization discussed here, and also requires an estimate for the norm of the discrete pseudoinverse which is outside the scope of this paper.
  Nonetheless, we expect that the presented results will be useful for extending
  it to pixel-driven methods for Radon transform inversion and obtaining
  convergence rates for the solution of
  the inverse problem that can directly be linked with the
  convergence rates for the discretization obtained here.
\end{remark}

Next we wish to consider the convergence behavior of the adjoint towards the backprojection. This does not require additional analysis since adjoint approximations have the same rates of convergence in the operator norm  to the  adjoint operator as the original approximation. So the statements of~\cref{thm:radon_convergence} concerning suitable  discretization strategies, and the corresponding convergence results can be transferred.

\begin{corollary}
\label{cor:backproj_convergence}
  If $\delta_s \to 0$, $\frac{\delta_\varphi}{\delta_s} \to 0$ and
  $\frac{\delta_x}{\delta_s }\to 0$, then
  $(\RadonApp)^*$ converges to $\Radon^*$ in
  operator norm for linear and continuous mappings
  $L^2(\Omega') \to L^2(\Omega)$.
  If additionally,
  $\delta_\varphi = \kronO(\delta_s^{1 + \epsilon})$ and
  $\delta_x = \kronO(\delta_s^{1 + \epsilon})$ for some
  $0 < \epsilon \leq 1/2$, then
  $\norm{ (\RadonApp)^* - \Radon^*} =
  \kronO(\delta_s^\epsilon)$ as $\delta_s \to 0$.
\end{corollary}
\begin{proof}
  This is a direct consequence of~\cref{thm:radon_convergence} as
  the norm of a linear, continuous operator between Hilbert spaces and
  the norm of its Hilbert space adjoint coincide.
\end{proof}

Note that the restriction $\epsilon\leq \frac{1}{2}$ is due to the
fact in general, the Radon transform for $f \in L^2(\Omega)$ generates
at least a regularity
$\modcont{\Radon f}(h,0) = \kronO(\abs{h}^{\frac12})$,
see~\cref{lem:radon_mod_continuity}. However, for functions $f$ whose
Radon transform admits higher regularity in terms of the modulus of
continuity, this restriction does not apply, as summarized in the
following theorem.
\begin{theorem} \label{thm:convergence_highregularity} Let
  $f \in L^2(\Omega)$ such that the modulus of
  continuity satisfies $\modcont{\Radon f}(h,0)=\kronO(\abs{h}^\epsilon)$ for
  some $\epsilon > 0$. If, additionally,
  $\delta_\varphi = \kronO(\delta_s^{1 + \epsilon})$ and
  $\delta_x = \kronO(\delta_s^{1 + \epsilon})$, then
  $\norm{\RadonApp f - \Radon f} =
  \kronO(\delta_s^\epsilon)$ as $\delta_s \to 0$.
  Moreover, for $g \in L^2(\Omega')$ with $\modcont{g}(h,0)
  =\kronO(\abs{h}^\epsilon)$ we have
  $\norm{(\RadonApp)^* g - \Radon^* g} =
  \kronO(\delta_s^\epsilon)$ as $\delta_s \to 0$.
\end{theorem}

\begin{proof}
  The first statement follows from the combination of
  Lemmas~\ref{lem:radon_detector_discretization},~\ref{lem:radon_detector_angular_discretization}
  and~\ref{lem:radon_image_discretization}, while the second is a
  consequence of
  Lemmas~\ref{lem:radon_detector_discretization_Backprojection},~\ref{lem:radon_detector_angular_discretization}
  and~\ref{lem:radon_image_discretization}, using again the fact that
  the norms of a linear, continuous operator and its adjoint coincide.
\end{proof}

\begin{remark}
  While the presented theory used the hat-shaped function
  $w_{\delta_s}(t) =\max(0,\delta_s-|t|)$, other profile functions are
  possible. The theory can be developed analogously for
  all Lipschitz continuous, non-negative $w_{\delta_s}$ which integrate
  to $\delta_s^2$, whose support is compact and whose translates
  with respect to integer multiples of $\delta_s$ sum up to
  the function that is constant $\delta_s$.
\end{remark}

\subsection{Radon transform with limited angle information}\label{Subsec:Limited_information}
While classical tomography uses information for the entire angular range ${[{-\pi ,\pi}[}$, some applications --- due to technical limitations --- have limited freedom in the angles from which projection can be obtained. In spite of the increased difficulty in performing tomography with restricted angular range, some practical procedures require reconstruction from such data.  In the following, we therefore  consider two types of incomplete angle information and show how the theory of pixel-driven Radon transforms extends to such situations. First, the limited angles situation is considered, where the discretization of the angular direction does not cover the entirety of $S^1$, but a finite union of open intervals, e.g., only angles between $\pm 70^\circ$.  Secondly, we consider the sparse angles situation, i.e., one discretizes only the space and offset dimension, while projections for finitely many fixed angles are considered.

\subsubsection{Limited angles}
\label{subsubsec:limited_angles}
In the following, we consider an angle set
$\AL \subset {[{-\pi,\pi}[}$ which corresponds to an open, non-empty
interval modulo $2\pi$ and satisfies $\AL \neq {[{-\pi,\pi}[}$. The
limited-angle Radon transform $\RadonLA$ is then the Radon transform
restricted to $\Omega_\AL' = {]{-1,1}[} \times \AL$, yielding a linear
and continuous mapping $\RadonLA: L^2(\Omega) \to L^2(\Omega_\AL')$ as
well as a corresponding adjoint. For the discretization of $\RadonLA$,
we can proceed analogously, but only need to discretize the angular
domain $\AL$ instead of the whole interval ${[{-\pi,\pi}[}$. With
$\varphi_1, \varphi_{Q} \in \RR$ chosen such that
${]{\varphi_1,\varphi_{Q}}[} = \AL \mod 2\pi$, let
$\varphi_2,\ldots \varphi_{Q-1} \in \RR$ be chosen such that
$\varphi_1 < \varphi_2 < \ldots < \varphi_{Q-1} < \varphi_{Q}$. With
$\varphi_0 = \varphi_1$ and $\varphi_{Q} = \varphi_{Q+1}$, the
corresponding $\Phi_q$, $q = 1,\ldots,Q$, form an a.e.~partition of
$\AL$. The corresponding %
discrete operators $\RadonPhi$ and $\RadonApp$ defined in~\eqref{eq:radon_detector_angle_discretization} and~\eqref{eq:radon_detector_angle_space_discretization} thus naturally map
$L^2(\Omega) \to L^2(\Omega_\AL')$, while a corresponding restriction of $\RadonS$ according to~\eqref{eq:radon_detector_discretization} leads to a mapping from $L^2(\Omega)$ to
$L^2(\Omega_\AL')$. Considering the $L^2$-norms on $\Omega'_\AL$
instead of $\Omega'$, i.e., integrating over $\AL$ instead of $S^1$,
we see that the statements of the
Lemmas~\ref{lem:radon_detector_discretization},~\ref{lem:radon_detector_discretization_Backprojection},
\ref{lem:radon_detector_angular_discretization}
and~\ref{lem:radon_image_discretization} remain true for these
modifications.  Consequently, we have the following theorem.

\begin{theorem}
  \label{thm:limited_angle_convergence}
  Considering $\RadonLA$ and $(\RadonLA)^*$ instead of $\Radon$
  and $\Radon^*$, respectively, the convergence results
  of~\cref{thm:radon_convergence}, \cref{cor:backproj_convergence}
  and~\cref{thm:convergence_highregularity}
  remain true.
\end{theorem}

Note that one can easily generalize the results to $\AL$ consisting of
finitely many intervals instead of just one: If
$\AL = \AL_1 \cup \ldots \cup \AL_I$ where each $\AL_i$ is an interval
of the above type and the $\AL_1,\ldots,\AL_I$ are pairwise disjoint,
then $L^2(\Omega_\AL')$ can be identified with
$L^2(\Omega_{\AL_1}') \times \cdots \times L^2(\Omega_{\AL_I}')$ and
$\RadonLA f$ can be identified with 
$(\Radon^{\AL_1}f, \ldots,
\Radon^{\AL_I}f)$. As~\cref{thm:limited_angle_convergence} can be applied to
every $\Radon^{\AL_i}$, the results also follow for $\RadonLA$.

\subsubsection{Sparse angles}
The Radon transform can also be defined for a finite angle set
$\mathcal{F} = \sett{\varphi_1, \ldots, \varphi_Q} \subset
{[{-\pi,\pi}[}$ for $\varphi_1 < \varphi_2 < \ldots <
\varphi_Q$. Denoting by
$\Omega_\mathcal{F}' = {]{-1,1}[} \times \mathcal{F}$, continuous
extension of~\eqref{equ:def_continuous_Radontrans} yields the linear
and continuous operator
$\RadonSA: L^2(\Omega) \to L^2(\Omega_\mathcal{F}')$, where
$L^2(\Omega_\mathcal{F}')$ is associated with the counting measure in
the angular direction, i.e.,
$\norm{g}^2 = \sum_{q=1}^Q \int_{]{-1,1}[} \abs{g(s,\varphi_q)}^2
\dd{s}$ for $g \in L^2(\Omega_\mathcal{F}')$.  Then,
equations~\eqref{eq:radon_detector_discretization},~\eqref{eq:radon_detector_angle_discretization}
and~\eqref{eq:radon_detector_angle_space_discretization} yield
respective (semi-)discrete sparse-angle operators $\RadonS$,
$\RadonPhi$ and $\RadonApp$, and since each $\varphi_q \in \Phi_q$, we
have $\RadonS = \RadonPhi$. Further, as each $\Phi_q$ is assigned
unit mass, it holds that $\Delta_q = 1$ for each $q = 1,\ldots,Q$.

However, since the sparse-angle Radon transform $\RadonSA$ is no longer a
restriction of the full transform $\Radon$, we cannot expect similar
approximation results in this situation. In particular, the smoothing
property of~\cref{lem:radon_mod_continuity} cannot be established for
$\RadonSA$.  Nevertheless, replacing $\mathcal{H}^1$-integration on
$S^1$ by $\mathcal{H}^0$-integration (i.e., summation) on
$\sett{\vartheta_1,\ldots,\vartheta_Q}$, the statements of
Lemmas~\ref{lem:radon_detector_discretization},~\ref{lem:radon_detector_discretization_Backprojection}
and~\ref{lem:radon_image_discretization} can still be obtained by
straightforward adaptation. This is sufficient to prove strong
operator convergence.

\begin{theorem} \label{thm:finite_angle_convergence}
  Let $\delta_s\to 0$ and $\frac{\delta_x}{\delta_s}\to 0$. Then, for
  any $f\in L^2(\Omega)$ and $g \in L^2(\Omega_\mathcal{F}')$ it holds that
  \[
    \lim_{\delta_s\to 0} \| \RadonSA f- \RadonApp f\| =0
    \qquad \text{as well as} \qquad
    \lim_{\delta_s \to 0} \norm{(\RadonSA)^* g - (\RadonApp)^* g} = 0.
  \]
  If $\delta_x = \kronO(\delta_s^{1+\epsilon})$ for some $\epsilon > 0$,
  then it holds for $f \in L^2(\Omega)$ with $\modcont{\RadonSA f}(h,0) = \kronO(\abs{h}^\varepsilon)$ and $g \in L^2(\Omega_\mathcal{F}')$ with
  $\modcont g(h,0) = \kronO(\abs{h}^\epsilon)$ that
  \[
    \| \RadonSA f- \RadonApp f\| = \kronO(\delta_s^\epsilon)
    \qquad \text{as well as} \qquad
    \norm{(\RadonSA)^* g - (\RadonApp)^* g} = \kronO(\delta_s^\epsilon).
  \]
\end{theorem}

\begin{proof}
  The combination of
  Lemmas~\ref{lem:radon_detector_discretization},~\ref{lem:radon_detector_discretization_Backprojection},~\ref{lem:radon_image_discretization}
  adapted to sparse angles and the fact that $\RadonS = \RadonPhi$
  yields, for $f \in L^2(\Omega)$ that
  \[
    \norm{\RadonSA f- \RadonApp f} \leq C \Bigl( \sup_{\abs{h} <
      \frac32 \delta_s} \modcont{\Radon f}(h,0) +\sqrt{1 +
      \frac{\delta_x}{\delta_s}} \frac{\delta_x}{\delta_s} \norm{f} \Bigr)
  \]
  and
  $g \in L^2(\Omega_\mathcal{F}')$ that
  \[
    \norm{(\RadonSA)^* g - (\RadonApp)^* g} \leq C \Bigl(
    \sup_{\abs{h} < \frac32\delta_s} \modcont g(h,0) + \sqrt{1 +
      \frac{\delta_x}{\delta_s}} \frac{\delta_x}{\delta_s} \norm{g}
    \Bigr).
  \]
  The first statement then follows from the fact that the modulus of
  continuity converges to zero for any $L^2$-function, while the
  second is an immediate consequence of the assumed rates.
\end{proof}

\section{The pixel-driven fanbeam transform}\label{Sec:Fanbeam}
Projection methods are  not limited to the parallel beam setting as some applications require different measurement and sampling approaches. One such different setting is the fanbeam setting that allows for an alternative version of tomography with a single-point source sending rays along non parallel lines to the detector. In the following, we present a discretization of the fanbeam transform following the same basic principle as used for the pixel-driven Radon transform and show convergence with analogous methods using the relation between the Radon transform and the fanbeam transform. 
\subsection{Definition and notation}\label{Subsec:Fanbeam_Definition} We consider the following geometry, see~\cref{Fig:fanbeam_geometry}:
We assume the density of a sample to be supported in the unit ball $B(0,1)$, that the distance from the emitter to the origin is $\RE>1$ and does not depend on the specific angle the source is placed in relation to the sample. Moreover, $\RD > \RE + 1$ denotes the distance from the source to the detector, while the total width $\DW$ of the detector is chosen such that all lines from the source passing through $B(0,1)$ are detected, which amounts to setting $\DW = 2\frac{R}{\sqrt{\RE^2 - 1}}$.

\usetikzlibrary{decorations.pathreplacing}
\usetikzlibrary{positioning,patterns}
\begin{figure}
\center
 \begin{tikzpicture}
 
\draw[->,thick,dashed] (0,0) -- (0,2.5) node [above,yshift=-0.05cm] {$\bold y$};
\draw[->,thick,dashed] (0,0) -- (2.5,0) node [right,xshift=-0.05cm] {$\bold x$};

\draw[thick] (0,0) circle (2cm );

\draw[thick,dashed] (1.69,1.69) -- (-1.75,-1.75);

\draw [very thick,xshift=0.1cm,yshift=0.1cm,fill=gray,rotate around={135:(1.69+0.15,1.69+0.15)}] (1.69- 2.83+0.15,1.69-0.15+0.15) rectangle (1.69+ 2.83+0.15,1.69+0.15+0.15);

\draw[thick,->,blue] (-1.75+1.5,-1.75) arc [radius=1.5cm,start angle=0,end angle=45] node [midway, left,yshift=-0.1cm,xshift=0.05cm]{$\alpha$};
\draw[thick,dashed] (-1.75,-1.75) -- (-1.75+2,-1.75) ;

\draw[->,thick]  (0,0) --(1.69,1.69) node [midway,rotate=45,xshift=-0.15cm,yshift=0.25cm]{$\vartheta(\alpha)^\perp$};

\draw[] %
 (1.69-2-0.2,1.69+2+0.2) -- (1.69+2+0.2,1.69-2-0.2) node [xshift=0.7cm, yshift=0.0cm] {detector};
 \draw[thick,->] (1.69,1.69) -- (1.69+1.69,1.69-1.69)node [midway,rotate=-45,xshift=-0cm,yshift=0.05cm,below]{$\vartheta(\alpha)$};

\draw[decorate,decoration={brace,amplitude=10pt},xshift=0.15cm,yshift=0.15cm] (1.69-2+0.3,1.69+2+0.3) -- (1.69+2+0.3,1.69-2+0.3) node [black,midway,xshift=0.475cm,yshift=0.3cm] {\footnotesize $\DW$};

\draw[decorate,decoration={brace,amplitude=10pt},xshift=-0.075cm,yshift=0.075cm] (-1.75,-1.75) -- (0,0) node [black,midway,xshift=-0.125cm,yshift=0.475cm] {\footnotesize $\RE$};

\draw[decorate,decoration={brace,amplitude=10pt},xshift=-2.3cm,yshift=2.3cm]       (-1.75,-1.75) -- (1.69,1.69) node [black,midway,xshift=0.2cm,yshift=0.3cm,xshift=-0.65cm] {\footnotesize $\RD$};
\draw[dashed] (-1.75,-1.75) -- (-1.75-2.2,-1.75+2.2);

\draw[very thick, red] (-1.75,-1.75) -- (1.69-1.82,1.69+1.82) node [near end,left,red,xshift=0.125cm,yshift=0.2cm] {$L( \xi,\alpha)$};

\draw[blue,thick,<-] %
(1.69,1.69) -- (1.69-1.82,1.69+1.82) node [midway,xshift=-0.2cm,yshift=-0.2cm] {$\xi$};

 \draw (1.69,1.69)  (-1.75,-1.75) node (first)[draw,solid,very thick,rotate=-45,fill=gray,xshift=0cm,yshift=-0.129cm] {\quad \quad} node [ left of =first, node distance =0.8cm,rotate=0,xshift=0.15cm,yshift=-0.15cm]{source};

  \end{tikzpicture}
  \caption{Geometry for the fanbeam transform.
    The line $L(\alpha,\xi)$ connects the source and the detector,
    both rotated by the angle $\alpha$, at detector offset $\xi$.  The
    values $\RD$ and $\RE$ denote the distances from the emitter to
    the detector and origin, respectively, while $W$ denotes the detector width.}
  \label{Fig:fanbeam_geometry}
\end{figure}
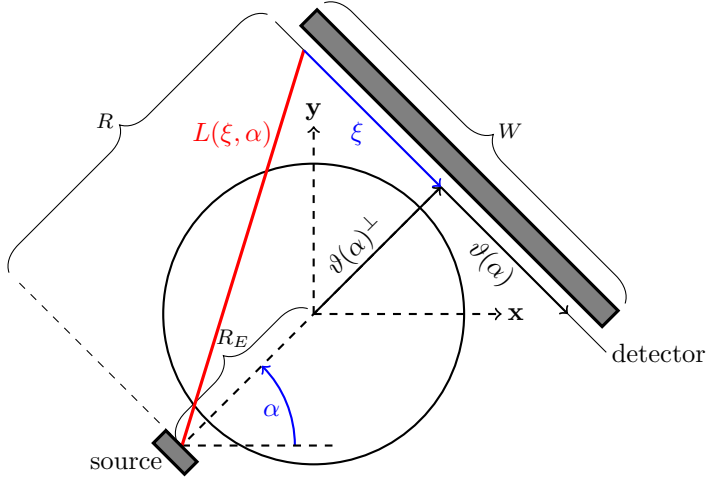

\begin{definition}
  The fanbeam transform %
  of a continuous $f: \RR^2 \to \RR$ with support compact within $B(0, \RE)$ is defined as
\begin{align} \label{equ:Def_Fanbeam_continuous}
\notag [\Fanbeam f](\xi, \alpha)&= \int_{\set{x \in \RR^2}{\frac{x\cdot \vartheta(\alpha) \RD}{x\cdot \vartheta(\alpha)^\perp+\RE}=\xi}} f(x)\dd{\mathcal{H}^1(x)}
\\&
=\sqrt{\xi^2+\RD^2}\int_\RR f\big(t(\xi \vartheta(\alpha)+\RD \vartheta(\alpha)^\perp)-\RE \vartheta(\alpha)^\perp\big)\dd{t},
\end{align}
where $\xi\in \RR$ %
is the detector offset and $\alpha \in {[{-\pi,\pi}[}$ denotes the angle between the shortest line connecting source and detector and the $\bold x$-axis. 
The adjoint operation
for $g: \RR \times S^1 \to \RR$ continuous with compact support
and $x \in B(0, \RE)$
is defined as
\begin{equation} \label{equ:Def_Fanbeam_continuous_adjoint}
[\Fanbeam^* g](x) = \int_{S^1}  \sqrt{\left (\frac{x\cdot\vartheta\RD }{x\cdot \vartheta^\perp+\RE}\right)^2+\RD^2} \frac{1}{x\cdot \vartheta^\perp+\RE} g\Big(\frac{x\cdot \vartheta\RD}{x\cdot \vartheta^\perp+\RE},\vartheta\Big)\dd\vartheta.
\end{equation}
\end{definition}

\begin{remark}
In the above definition, the set $\set{x \in \RR^2}{\frac{x\cdot \vartheta(\alpha) \RD}{x\cdot \vartheta(\alpha)^\perp+\RE}=\xi}$ %
describes the line from the source to detector at offset $\xi$ where both are rotated by $\alpha$.

As it is also the case for the Radon transform, the adjoint corresponds to an integral over all $L(\xi,\alpha)$ passing through $x$. In this context, we note that for a fixed $\vartheta = \vartheta(\alpha)$ and $x \in B(0, \RE)$, the detector offset $\xi$ and the integration variable $t$ in~\eqref{equ:Def_Fanbeam_continuous} can be expressed as
\[
  t = \frac{x \inprod \vartheta^\perp + \RE}{R},
  \qquad
  \xi = \frac{x \inprod \vartheta}{t} =  \frac{x \inprod \vartheta R}{x \inprod \vartheta^\perp + \RE}.
\]
With the change of coordinates $x = t(\xi \vartheta + R\vartheta^\perp) - \RE \vartheta^\perp$ with transformation determinant $\frac1{Rt} = \frac{1}{x \inprod \vartheta^\perp + R_E}$, the operator $\Fanbeam^*$ in~\eqref{equ:Def_Fanbeam_continuous_adjoint} can easily be seen to be the formal adjoint of $\Fanbeam$ in~\eqref{equ:Def_Fanbeam_continuous} with respect to the $L^2$ scalar product.
\end{remark}

\begin{remark} It can also be observed that the fanbeam transform is a re\-para\-me\-tri\-za\-tion of the Radon transform according to
\begin{equation}\label{equ:Connection_Radon_fanbeam_arguments}
  \quad [\Fanbeam f] (\xi,\alpha)=[\Radon f] %
  (s,\varphi)%
 \quad \text{for} \quad
 \left( s= \frac{\xi \RE}{\sqrt{\xi^2+\RD^2}}, \quad \varphi= \alpha-\arctan\Big(\frac{\xi}{\RD}\Big)\right).
\end{equation}
In particular, $(\xi,\alpha)\mapsto(s,\varphi)$ is a diffeomorphism
between $\RR \times S^1$ and ${]{-\RE,\RE}[} \times S^1$.
\end{remark}

This different parametrization  also affects the sampling strategies, and thus, a suitable discretization of parameters and corresponding discrete image and  sinogram spaces must be considered. For this purpose, $Q$ angles $\alpha_1< \dots< \alpha_Q\in [-\pi,\pi[$ and an equidistant grid of $P$ detector offsets $\xi_1,\dots, \xi_P\in {]{-\frac{\DW}{2},\frac{\DW}{2}}[}$ with $\xi_p=\frac{\DW}{P}\bigl(p-\frac{(P+1)}{2}\bigr)$ are considered. We use $\Xi_p=\xi_p+{[{-\frac{\delta_\xi}{2},\frac{\delta_\xi}{2}}[}$ and $\Phi_q={[{\frac{\alpha_q+\alpha_{q-1}}{2},\frac{\alpha_q+\alpha_{q+1}}{2}}[}$ such that $(\Xi_p\times \Phi_q)_{pq}$ is a partition of the sinogram space, where $\delta_\xi=\frac{\DW}{P}$ is the degree of detector discretization, $\delta_\alpha=\max_{q=1,\ldots,Q} \alpha_{q+1} - \alpha_q$ denotes the angular discretization width and  $\Delta_q = (\alpha_{q+1} - \alpha_{q-1})/2$ denotes again the length of $\Phi_q$. Moreover, the discrete sinogram space $V$ is the space of functions on the grid $\{\xi_1,\dots,\xi_P\} \times \{\alpha_1,\dots,\alpha_Q\}$ equipped with the norm on $L^2(\RR \times S^1)$ as in~\eqref{equ:def_discrete_image_sinogram_spaces}.

Analogous to the Radon transform case,
the fanbeam transform is first approximated by replacing the line integral by an area integral resulting in
\begin{align}\label{equ:Def_Fanbeam_convolution_approximation}
[\FanbeamHat f](\xi,\alpha) &= \frac{\sqrt{\xi^2+\RD^2}}{\delta_\xi^2} \int_{\RR} w_{\delta_\xi}(\tau-\xi)\frac{[\Fanbeam f](\tau,\alpha)}{\sqrt{\tau^2+\RD^2}} \dd{\tau}
\\
&=\frac{\sqrt{\xi^2+\RD^2} }{\delta_\xi^2}\int_{\Omega} w_{\delta_\xi}\Big(\frac{x\cdot \vartheta(\alpha)\RD }{x\cdot \vartheta^\perp(\alpha)+\RE}-\xi\Big) \frac{f(x)}{x\cdot \vartheta^\perp(\alpha)+\RE}\dd{x}, \notag 
\end{align}
where again $w_{\delta_\xi}(\tau) =\max(0,\delta_{\xi}-|\tau|)$. Observe that we
weight the fanbeam transform with $\frac1{\sqrt{\tau^2 + R^2}}$ inside the integral with respect to $\tau$ which is compensated by $\sqrt{\xi^2 + R^2}$ outside the integral. This turns out to be advantageous in the subsequent analysis. Other choices are, of course, possible and require only minor adaptations.

The image to transform is again given on a discrete $N\times M$ grid with discretization width $\delta_x > 0$, and $x_{ij}$,  $X_{ij}$
as described in \Cref{Subsec:Radon_Def}. However, we additionally assume that the support of $f$ is such that whenever $|\supp f \cap X_{ij}| > 0$, then $\abs{x_{ij}} < \RE$, i.e., the centers of the pixels which contribute to the discrete fanbeam transform are contained in the ball $B(0,\RE)$, and adapt the space $U$ according to %
\[
  U = \linspan\set{\chi_{X_{ij}}}{\abs{x_{ij}} < \RE, \ i=1,\ldots,N,
    \ j=1,\ldots,M}.
\]
Then, performing the same discretization steps as for the Radon transform, i.e., using $f_{\delta_x}$ as discretization of $f$ according to~\eqref{eq:discretization_image}, extrapolation from $(\xi_p,\alpha_q)$ onto $\Xi_p\times \Phi_q$ and application of $\FanbeamHat$, yields \begin{multline}\label{equ:Def_Fanbeam_approximation_complete}
[\FanbeamApp f] (\xi,\alpha)= \frac{\delta_x^2}{\delta_\xi^2}  \sum_{p=1}^P \sum_{q=1}^Q \chi_{\Xi_p}(\xi)\chi_{\Phi_q}(\alpha)\sqrt{\xi_p^2+\RD^2} 
\\ 
\cdot \sum_{i,j=1}^{N,M} w_{\delta_\xi}\Big(\frac{x_{ij}\cdot \vartheta_q \RD  }{x_{ij}\cdot \vartheta_q^\perp+\RE}-\xi_p\Big) \frac{f_{ij}}{x_{ij}\cdot \vartheta^\perp_q+\RE},
\end{multline}
where $f_{ij} = \frac1{\delta_x^2} \int_{X_{ij}} f(x) \dd{x}$ as well as
\begin{multline}
[(\FanbeamApp)^*g](x)
=  \sum_{i,j=1}^{N,M} \chi_{X_{ij} }(x) \sum_{q=1}^Q
\frac{\Delta_q}{\delta_\xi}
\sum_{p=1}^P   w_{\delta_\xi}\Big(\frac{x_{ij} \cdot \vartheta_q \RD}{x_{ij} \cdot \vartheta^\perp_q+\RE}-\xi_p\Big)\\
\cdot \frac{\sqrt{\xi^2_p+\RD^2}}{x_{ij} \cdot \vartheta^\perp_q+\RE}
g_{pq},
\end{multline}
where $g_{pq} = \frac1{\delta_\xi \Delta_q} \int_{\Xi_p} \int_{\Phi_q} g(\xi,\alpha) \dd{\xi} \dd{\alpha}$.
Switching to the fully discrete setting by associating
elements of $U$ and $V$, respectively, in terms of their coefficients
gives $\FanbeamDisc: U \to V$ according to
\begin{equation} \label{equ:def_discrete_fanbeam}
[\FanbeamDisc f]_{pq}= \frac{\delta_x^2}{\delta_\xi^2}\sqrt{\xi_p^2+\RD^2}\sum_{i,j = 1}^{N,M} w_{\delta_\xi}\Big(\frac{ x_{ij}\cdot \vartheta_q \RD}{x_{ij}\cdot \vartheta^\perp_q+\RE}-\xi_p\Big) \frac{f_{ij}}{x_{ij}\cdot \vartheta^\perp_q+\RE},
\end{equation}
whose adjoint $\FanbeamDisc^*\colon V \to U $ reads as %
\begin{equation}\label{equ:def_discrete_fanbeam_backprojection} 
[\FanbeamDisc^*g]_{ij}=\sum_{q=1}^Q \frac{\Delta_q}{\delta_\xi} \sum_{p=1}^P   w_{\delta_\xi}\Big(\frac{x_{ij} \cdot \vartheta_q \RD}{x_{ij} \cdot \vartheta^\perp_q+\RE}-\xi_p\Big)
  \frac{\sqrt{\xi^2_p+\RD^2} }{x_{ij} \cdot \vartheta^\perp_q+\RE} g_{pq}.
\end{equation}
Note that in these discretizations, the distance between the source
and $x_{ij}$ projected to the shortest line connecting source and
detector, i.e., $x_{ij} \inprod \vartheta_q^\perp + \RE$, plays a
major role. This is because this distance describes the width of the
fan associated with $\Xi_p$ in the point $x_{ij}$, which is, by
construction, bounded from below and enters into the discrete fanbeam
transform in form of an inverse weight as well as a rescaling of the
hat function $w_{\delta_\xi}$.

\subsection{Convergence analysis}\label{Subsec:Fanbeam_Convergence}
The convergence analysis follows in broad strokes the approach in  \Cref{Subsec:Radon_Convergence}, using similar lemmas though some details in the proofs need to be adjusted. In the following, let $\Omega = B(0,1)$ and $\Omega' = {]{-\frac{\DW}2, \frac{\DW}{2}}[} \times S^1$. Further, assume that $\Omega$ is contained in the union of all pixels and that $0 < \delta_x < \frac1{\sqrt{2}}(\RE - 1)$ such that whenever $B(0,1) \cap X_{ij} \neq \emptyset$, we have $\abs{x_{ij}} < \RE$.

For technical reasons we consider the operator $\myG\colon L^2(\Omega)\to L^2(\Omega')$ with

\begin{equation}
[\myG f] (\xi,\alpha)= \int_\RR f\bigl(t(\xi \vartheta(\alpha)+\RD \vartheta^\perp(\alpha))-\RE \vartheta^\perp(\alpha)\bigr)\dd{t},
\end{equation}
 i.e., the operator $\Fanbeam$ in \eqref{equ:Def_Fanbeam_continuous} without the factor $\sqrt{\xi^2+\RD^2}$.
In particular, $\mM\myG =\Fanbeam$ for the continuously invertible multiplication operator $\mM\colon L^2(\Omega')\to L^2(\Omega')$ according to $[\mM g](\xi,\alpha)=\sqrt{\xi^2+\RD^2}g(\xi,\alpha)$ for $(\xi,\alpha) \in \Omega'$. We will first show convergence for 
\begin{multline} \label{equ:Def_Fanbeam_Mod_App}
[\myGApp f](\xi,\alpha)=\frac{1}{\delta_\xi^2} \sum_{p=1}^P\sum_{q=1}^Q \chi_{\Xi_p}(\xi)\chi_{\Phi_q}(\alpha) \sum_{i,j=1}^{N,M} w_{\delta_\xi}\Big(\frac{x_{ij} \cdot \vartheta_q \RD}{x_{ij} \cdot \vartheta^\perp_q+\RE}-\xi_p\Big)
\\
 \cdot \frac{\int_{X_{ij}}f(x)\dd{x}}{x_{ij} \cdot \vartheta^\perp_q+\RE}
\end{multline}
  towards $\myG$. 
  Then, writing $\mM_{\delta_\xi} \myGApp=\FanbeamApp$,
  where
  \[
    [\mM_{\delta_\xi} g] (\xi,\alpha)=g(\xi,\alpha)\sum_{p=1}^P\chi_{\Xi_p}(\xi)\sqrt{\xi^2_p+\RD^2}\] is a piecewise constant version of $\mM$, will eventually enable us to prove convergence of $\FanbeamApp$ to $\Fanbeam$.

We again require an estimate on the modulus of continuity for the fanbeam transform, which we obtain by pulling back to the Radon transform, and to do so, we require an additional result for the Radon transform that is interesting in its own right. The proof of the following lemma can be found in~\cref{app:proof}.

\begin{lemma} \label{lem:radon_angle_mod_continuity}
  Let $f\in L^2(\Omega)$ and $g=\Radon f$. Then, the modulus of continuity for
  $g$ satisfies $\modcont{g}(0,\gamma)^2 \leq c \abs{\gamma \log(\abs{\gamma})}$
  for each $\abs{\gamma} \leq \frac{\pi}{4}$ and
  some constant $c>0$ independent of $\gamma$ and $f$. 
\end{lemma}

This enables us to derive estimates for the modulus of continuity for the fanbeam transform and $\myG$, as a change of offset in the fanbeam transform corresponds to a change in offset and angle argument of the Radon transform.

\begin{lemma} \label{lem:fanbeam_mod_continuity} 
  Let $f\in L^2(\Omega)$, $g = \Fanbeam f$, $\bar g = \myG f$ %
  and $\delta_\xi \leq 2(\sqrt{2} - 1)$. Then,
  \begin{equation}
    \int_{|t|\leq \delta_\xi} \modcont{g}(t,0)^2\dd{t}\leq c \delta_\xi ^2 |\log(\delta_\xi) | \|f\|^2,
    \quad
    \int_{|t|\leq \delta_\xi} \modcont{\bar g}(t,0)^2\dd{t}\leq c \delta_\xi ^2 |\log(\delta_\xi) | \|f\|^2,
    \label{eq:fanbeam_modcont_integral_est}    
  \end{equation}
  for a constant $c>0$ independent of $\delta_\xi$ and $f$. This constant can be chosen to stay bounded for $R$ bounded and $\RE$ bounded away from $1$.
\end{lemma}

\begin{proof}
We use the relation $[\Fanbeam f](\xi,\alpha)=[\Radon f](s,\varphi)$ in \eqref{equ:Connection_Radon_fanbeam_arguments} and the notation $s = s(\xi)$, $\varphi = \varphi(\xi, \alpha)$ to compute
\begin{multline}\label{eq:fanbeam_modcont_radon}
\int_{|t|\leq \delta_\xi} \modcont g(t,0)^2 \dd{t} 
\\ 
= \int_{|t|\leq \delta_\xi}\int_\RR \int_{[-\pi,\pi[} \bigl|[\Radon f]\bigl(s(\xi+t),\varphi(\xi+t, \alpha) \bigr) -[\Radon f]\bigl(s(\xi),\varphi(\xi, \alpha) \bigr) \bigr|^2  \dd\alpha \dd \xi\dd{t}.
\end{multline}
Note that since $\supp f$ is essentially contained in $B(0,1)$ and $\delta_\xi \leq 2(\sqrt{2} - 1) \leq \frac{\DW}2$, the support of the integrand is essentially contained in ${]{-W,W}[} \times S^1$ for each fixed $\abs{t} \leq \delta_{\xi}$.
We now consider the transformation $T: (\xi,\alpha,t) \mapsto (s,\varphi,h)$ with $h=s(\xi+t)-s(\xi)$ which is a diffeomorphism mapping $\RR \times S^1 \times {]{-\delta_\xi, \delta_\xi}[}$ to the set
\[
  \Lambda_{\delta_\xi} = \bigset{(s,\varphi,h)}{\abs{s} < \RE, \ \varphi \in {[{-\pi,\pi}[}, \
  s\bigl(\xi(s) - \delta_\xi\bigr) < s + h < s\bigl(\xi(s) + \delta_\xi\bigr)},
\]
where $\xi(s) = R\frac{s}{\sqrt{\RE^2 - s^2}}$ denotes the inverse of $\xi \mapsto s(\xi)$. Since $\RE \leq R$, one easily deduces that the derivative of $\xi \mapsto s(\xi)$ satisfies $\abs{s'(\xi)} \leq 1$ such that $\Lambda_{\delta_\xi} \subset {]{-\RE,\RE}[} \times S^1 \times {]{-\delta_\xi,\delta_\xi}[}$. Further, the transformation determinant of $T$ is given by
\[
  \abs{\det \nabla T(\xi,\alpha,t)} = \frac{\RE^2 R^4}{\bigl(\xi^2 + R^2\bigr)^{3/2} \bigl((\xi + t)^2 + R^2\bigr)^{3/2}},
\]
which is bounded from above by $1$, again since $\RE \leq R$. For $\abs{\xi} < \frac{W}2$ and $\abs{t} < \frac{W}2$, we obtain the lower bound $\abs{\det \nabla T(\xi,\alpha,t)}
\geq \RE^6 /\bigl( W^2 + R^2 \bigr)^3  > 0$ which holds in particular on the
essential support of the integrand in~\cref{eq:fanbeam_modcont_radon}.
Thus, denoting by $\gamma(s,\varphi,h) = \varphi\bigl(\xi(s) + t(s,h), \alpha(s,\varphi) \bigr) - \varphi$ where $\alpha(s,\varphi) = \varphi + \arctan\bigl(\frac{\xi(s)}{R} \bigr)$ and $t(s,h) = \xi(s + h) - \xi(s)$, we get, for some $c > 0$ that
\begin{align} \notag
&\int_{|t|\leq \delta_\xi} \modcont g(t,0)^2 \dd{t}
                        = \int_{\Lambda_{\delta_\xi}} \frac{|[\Radon f] (s+h,\varphi +\gamma(s,\varphi,h))-[\Radon f] (s,\varphi)|^2}{\bigabs{\det \nabla T\bigl(\xi(s), \alpha(s,\varphi), t(s,h)\bigr)}} \dd{(s,\varphi,h)}
\\ \notag 
 &\leq c \int_{|h| \leq \delta_{\xi}} \int_\mathbb{R} \int_{[-\pi,\pi[}  |[\Radon f](s+h,\varphi +\gamma(s,\varphi,h))-[\Radon f] (s,\varphi +\gamma(s,\varphi,h))|^2\dd{\varphi}\dd{s}\dd{h}
\\
  & \quad +c %
    \int_{\Lambda_{\delta_{\xi}}} 
    |[\Radon f] (s,\varphi +\gamma(s,\varphi,h))-[\Radon f] (s,\varphi)|^2
    \dd{(s,\varphi,h)}.
    \label{eq:fanbeam_modcont_1}
\end{align}
The first integral does not change when $\varphi + \gamma(s,\varphi,h)$ is replaced by $\varphi$ and thus amounts to $\int_{\abs{h} \leq \delta_\xi} \modcont{\Radon f}(h,0)^2 \dd{h}$. For the second integral, which only needs to be considered for $\abs{s} < 1$, we change the coordinates according to $(s,\varphi,h) \mapsto (s,\varphi,\gamma)$ where one computes $\gamma  = \gamma(s,\varphi,h) = \arctan \bigl(\frac{\xi(s)}R \bigr) - \arctan\bigl(\frac{\xi(s+h)}R\bigr)$. Clearly,
this is a diffeomorphism between
$\Lambda_{\delta_\xi}$
and
  \begin{align*}
    \Lambda'_{\delta_\xi} = \bigl\{(s,\varphi,\gamma)\,
    \colon\,&\abs{s} < \RE,\ \varphi \in [-\pi,\pi[,
    \\
&    \arctan\bigl(\tfrac{\xi(s) - \delta_\xi}R\bigr) < \arctan\bigl(\tfrac{\xi(s)}R\bigr) - \gamma < \arctan\bigl(\tfrac{\xi(s) + \delta_\xi}R\bigr)\bigr\}.
  \end{align*}
Also here, one can see that $\Lambda_{\delta_\xi}' \subset {]{-\RE,\RE}[} \times S^1 \times {]{-\delta_{\xi}', \delta_{\xi}'}[}$
where $\delta_{\xi}' = 2\arctan(\frac{\delta_\xi}2)$. The transformation determinant can further be computed as $\frac{R^2}{\xi(s+h)^2 + R^2} \frac{\RE^2}{(\RE^2 - (s+h)^2)^{3/2}}$ which is, for
$(s,\varphi,h) \in \Lambda_{\delta_\xi}$ with $\abs{s} < 1$, bounded with a positive lower bound. Hence, we can estimate, for some $c' > 0$,
\begin{multline} 
  \int_{\Lambda_{\delta_{\xi}}} 
  |[\Radon f] (s,\varphi +\gamma(s,\varphi,h))-[\Radon f] (s,\varphi)|^2
  \dd{(s,\varphi,h)}
  \\ 
  \leq c' \int_{\abs{\gamma} \leq \delta_{\xi}'} \int_{\abs{s} \leq 1}
  \int_{[-\pi,\pi[} |[\Radon f] (s,\varphi +\gamma)-[\Radon f] (s,\varphi)|^2
  \dd{\varphi} \dd{s} \dd{\gamma}
  \\ \label{eq:fanbeam_modcont_2} = c' \int_{\abs{\gamma} \leq \delta_{\xi}'} \modcont{\Radon f}(0,\gamma)^2 \dd{\gamma}.
\end{multline}
Since $\delta_{\xi} \leq 2(\sqrt{2} - 1)$, we have
$\delta_\xi' \leq \frac{\pi}4$, so
\cref{lem:radon_angle_mod_continuity} can be applied for each $\abs{\gamma} \leq \delta_{\xi}'$. Combining this, \cref{lem:radon_mod_continuity} as well as \cref{eq:fanbeam_modcont_1} and \cref{eq:fanbeam_modcont_2}, and possibly enlarging $c$ yields
  \begin{align*}
  \int_{\abs{t} \leq \delta_\xi} \modcont{g}(t,0)^2 \dd{t} &\leq c \Bigl( \int_{\abs{h} \leq \delta_\xi} \abs{h} \dd{h} + \int_{\abs{\gamma} \leq \delta_\xi'} \abs{\gamma} \abs{\log(\abs{\gamma})} \dd{\gamma} \Bigr) \norm{f}^2 \\ &\leq c \bigl(
  \delta_\xi^2 + (\delta_\xi')^2 + (\delta_\xi')^2\abs{\log(\delta_\xi')} \bigr) \norm{f}^2.
  \end{align*}
As $\delta_\xi \leq 2(\sqrt{2} - 1) < 1$,
we can find $c'' > 0$ independent of $\delta_\xi$ such that $1 \leq c'' \abs{\log(\delta_\xi)}$. With $\delta_{\xi}' \leq \delta_\xi$ and possibly enlarging $c$ once more, we arrive 
at the first estimate in \cref{eq:fanbeam_modcont_integral_est}.  
Concerning the second estimate,
we observe that the function given by $\mu(\xi)=\frac {1} {\sqrt{\xi^2+\RD^2}}$ is  %
bounded and Lipschitz continuous in 
${]{-\DW,\DW}[}$. Since we have
\begin{multline} \label{equ:proof_mod_cont_equivalence}
  [\myG f](\xi+t,\alpha)-[\myG f](\xi,\alpha)= \mu (\xi+t)[\Fanbeam f](\xi+t,\alpha)-\mu(\xi)[\Fanbeam f](\xi,\alpha)
\\=\left(\mu(\xi+t)-\mu(\xi)\right)[\Fanbeam f](\xi+t,\alpha)-\mu(\xi)\left([\Fanbeam] f(\xi,\alpha)- [\Fanbeam f] (\xi+t,\alpha)\right),
\end{multline}
we can find a $c''' > 0$ such that $\modcont{\bar g}(t,0)^2 \leq c''' (t^2 \norm{g}^2 + \modcont{g}(t,0)^2)$. Integration over $\abs{t} \leq \delta_\xi$, estimating $\norm{g}^2 \leq \norm{\Fanbeam}^2 \norm{f}^2$ and possibly enlarging $c$ then leads to the second estimate of~\eqref{eq:fanbeam_modcont_integral_est}.

Finally, we observe that when $R$ is bounded and $\RE$ is bounded away from $1$, then $W$ stays bounded which enables us to choose $c$ in each of the above steps in a bounded way.
\end{proof}

With these results we can consider a discretization in the offset parameter and approximation via an area integral resulting in 
\begin{align}
[\myGS f](\xi,\alpha)&= \frac{1}{\delta_\xi^2} \sum_{p=1}^P \chi_{\Xi_p}(\xi)\int_\RR w_{\delta_\xi}(t-\xi_p)[\myG f](t,\alpha)\dd{t} \notag
\\
&=\frac{1}{\delta_\xi^2} \sum_{p=1}^P \chi_{\Xi_p}(\xi)\int_\Omega w_{\delta_\xi}\Big(\frac{x \cdot \vartheta(\alpha) \RD}{x \cdot \vartheta^\perp(\alpha)+\RE}-\xi_p\Big)\frac{f(x)}{x \cdot \vartheta^\perp(\alpha)+\RE}\dd{x}.
\end{align}
\begin{lemma} \label{lem:F_detector_discretization}
  For $\delta_\xi \leq \frac43 (\sqrt{2} - 1)$, we have $\|\myG-\myGS\| \leq c \sqrt{\delta_\xi}  |\log(\delta_\xi)|^\frac{1}{2}$
  with the constant $c$ being independent of $\delta_\xi$.
\end{lemma}
\begin{proof}
  The proof works out in a way that is analogous to the proof of~\cref{lem:radon_detector_discretization} up to the last line \cref{equ:Proof_lastequation_2.2}, leading to
\begin{equation} \label{equ:proof_modulus_to_estimate}
  \|\myG f-\myGS f\|^2 \leq \frac{1}{\delta_\xi} \int_{|h|\leq \frac{3}{2}\delta_\xi} \modcont{\myG f}(h,0)^2 \dd{h}
    \leq \tfrac94 c \delta_\xi \abs{\log(\tfrac32 \delta_\xi)}
  \norm{f}^2,
\end{equation} 
  the latter since $\tfrac32 \delta_\xi \leq 2 (\sqrt{2} - 1)$ and
  consequently, \cref{lem:fanbeam_mod_continuity} can be applied. This
  implies the desired estimate.
\end{proof}

\begin{lemma}
\label{lem:fanbeam_detector_discretization_Backprojection}
For $g \in L^2(\Omega')$ the adjoint of $\myGS$ is 
\begin{equation*}
[(\myGS)^*g](x)=\frac{1}{\delta_\xi^2}\sum_{p=1}^P \int_{S^1} w_{\delta_\xi}\left(\frac{x \cdot \vartheta \RD}{x \cdot \vartheta^\perp +\RE}-\xi_p \right ) \frac{1}{x \cdot \vartheta^\perp +\RE} \int_{\Xi_p}g(\xi,\vartheta) \dd{\xi} \dd{\vartheta},
\end{equation*}
and the approximation error for the adjoint applied to $g$ can be estimated by
\begin{equation}
\norm{(\myGS)^* g - \myG^* g}\leq C \sup_{\abs{h} < \frac32 \delta_\xi} \modcont g(h,0).
\end{equation}
\end{lemma}

\begin{proof}
The representation of the adjoint of $\myGS$ can readily be shown via simple computation.
The proof for the error estimate works analogously to \cref{lem:radon_detector_discretization_Backprojection}, though some computations are more involved, thus we only give a brief overview of the intermediary results of the proof. Using the Cauchy--Schwarz inequality, Jensen's inequality as well as $\frac{1}{x \cdot \vartheta^\perp +\RE} \leq \frac{1}{\RE-1}$, one gets
\begin{align*}
\norm{(\myGS)^* g - \myG^* g}^2
&\overset{}{\underset{}{\leq}}
\frac{2 \pi }{|\RE-1|^2}\int_\Omega\int_{S^1} \sum_{p=1}^P \frac{1}{\delta_\xi}w_{\delta_\xi}\left(\frac{x \cdot \vartheta \RD}{x \cdot \vartheta^\perp +\RE}-\xi_p \right )
\\ &\qquad \qquad \qquad  \cdot\left | \frac{1}{\delta_\xi} \int_{\Xi_p} g\left(\frac{x \cdot \vartheta \RD}{x \cdot \vartheta^\perp +\RE}, \vartheta\right)- g(\xi,\vartheta) \dd{\xi} \right|^2 \dd \vartheta  \dd{x}.
\end{align*} 
One can apply the Cauchy--Schwarz inequality once more to %
estimate by the square of the term in the
innermost integral. Moreover, note that $\big|\frac{x \cdot \vartheta \RD}{x \cdot \vartheta^\perp +\RE}-\xi \big|\leq \frac{3}{2}\delta_\xi$ for $\xi \in \Xi_p$ and $w_{\delta_\xi}\bigl(\frac{x \cdot \vartheta \RD}{x \cdot \vartheta^\perp +\RE}-\xi_p \bigr )\neq 0$, thus one substitutes $\xi=\frac{x \cdot \vartheta \RD}{x \cdot \vartheta^\perp +\RE}+h$ for $h$, and $x=t(\xi\vartheta+\RD \vartheta^\perp)-\RE \vartheta^\perp$ for $(\xi,t)$ with transformation determinant $Rt$ leading to
\begin{multline*}
  \norm{(\myGS)^* g - \myG^* g}^2 \leq  \frac{2 \pi R}{|\RE-1|^2\delta_\xi }
  \int_\RR
  \int_{0}^1 t \dd t \int_{S^1}  \sum_{p=1}^P \frac{1}{\delta_\xi}w_{\delta_\xi}\left(\xi-\xi_p \right ) \\ \cdot \int_{|h|\leq \frac{3}{2}\delta_\xi} \left|g\left(\xi+h, \vartheta\right)- g(\xi,\vartheta)\right|^2 \dd{h}\dd \vartheta \dd{\xi} . 
\end{multline*}
Using that $\sum_{p=1}^P \frac {1} {\delta_\xi} w_{\delta_\xi}(\cdot-\xi_p) \leq1$ and %
interchanging integration order yields
\begin{equation*}
\norm{(\myGS)^* g - \myG^* g}^2 
\leq \frac{\pi R}{|\RE-1|^2 \delta_\xi} \int_{|h|\leq \frac{3}{2}\delta_\xi} \modcont g(h,0)^2 \dd{h}.
\end{equation*}
This implies the statement of the lemma.
\end{proof}

Next, we also discretize the angle dimension leading to
\begin{multline}
[\myGPhi f](\xi,\alpha)=\frac{1}{\delta_\xi^2} \sum_{p=1}^P\sum_{q=1}^Q \chi_{\Xi_p}(\xi)\chi_{\Phi_q}(\alpha) \int_\RR w_{\delta_\xi}(t-\xi_p) [\myG f](t,\alpha_q) \dd{t}
\\
=\frac{1}{\delta_\xi^2} \sum_{p=1}^P\sum_{q=1}^Q \chi_{\Xi_p}(\xi)\chi_{\Phi_q}(\alpha)\int_\Omega w_{\delta_\xi}\Big(\frac{x \cdot \vartheta_q \RD}{x \cdot \vartheta^\perp_q+\RE}-\xi_p\Big)
\frac{f(x)}{x \cdot \vartheta^\perp_q+\RE} \dd{x}.
\end{multline} 
\begin{lemma} \label{lem:F_detector_angular_discretization}
  For $\delta_\xi \leq 1$, we have $\|\myGPhi- \myGS\|\leq c \frac{\delta_\alpha}{\delta_\xi}$ for some constant $c>0$ independent of $\delta_\xi$ and $\delta_\alpha$ that remains bounded for $R$ bounded and $\RE$ bounded away from $1$.
\end{lemma}
\begin{proof}
  The proof %
  can be done
  analogous to the one for \cref{lem:radon_detector_angular_discretization}
  and leads to the estimation of
  \[
    \Bigl|  w_{\delta_\xi}\Bigl (\frac{x \cdot \vartheta_q \RD}{x \cdot \vartheta^\perp_q+\RE}-\xi_p\Bigr) \frac{1}{x \cdot \vartheta^\perp_q+\RE} -  w_{\delta_\xi}\Bigl(\frac{x \cdot \vartheta \RD}{x \cdot \vartheta^\perp+\RE}-\xi_p\Bigr) \frac{1}{x \cdot \vartheta^\perp+\RE}\Bigr| %
\] %
for fixed $p$ and $q$, $\vartheta \in S^1$ and $x \in \Omega$.
For this purpose, we see that the absolute value of the weak derivative of $\vartheta \mapsto \frac1{x \inprod \vartheta^\perp + \RE} w_{\delta_\xi}(\frac{x \inprod \vartheta R}{x \inprod \vartheta^\perp + \RE} - \xi_p)$
can be bounded by $c' = \frac{(R + 1)^2}{(\RE - 1)^3}$ for $x \in \Omega$ on the stripe
$\abs{\frac{x \inprod \vartheta R}{x\inprod \vartheta^\perp + \RE} - \xi_p} \leq \delta_\xi$ and vanishes for all other $x \in \Omega$. The area of the stripe can, in turn, roughly be estimated by $R\delta_\xi$, such that with the approach in the proof of \cref{lem:radon_detector_angular_discretization}, the above function obeys the bound $c' \abs{\vartheta - \vartheta_q}$ and its
integral over $\Omega$ can be bounded by $c'R \delta_\xi \abs{\vartheta - \vartheta_q}$. This leads to the estimate
\[
  \norm{[\myGPhi f](\xi_p, \placeholder) - [\myGS f](\xi_p, \placeholder)}^2
  \leq \frac{(c')^2 R}{\delta_\xi^3} \Bigl(\sum_{q=1}^Q \int_{S^1} \chi_{\Theta_q}(\vartheta) \abs{\vartheta - \vartheta_q}^2  \dd{\vartheta}\Bigr) \norm{f}^2,
\]
and following the proof of \cref{lem:radon_detector_angular_discretization},
to the desired statement.
Observe that in particular, $c' > 0$ stays bounded under the stated conditions,
such that $c > 0$ can also be chosen to remain bounded.
\end{proof}

Finally, we also discretize $x\in\Omega$ leading to the discrete $\myGApp$ in \eqref{equ:Def_Fanbeam_Mod_App} and consider the corresponding discretization error.
\begin{lemma}
\label{lem:fanbeam_image_discretization}
For $\delta_\xi \leq 1$ and $\delta_x < \sqrt{2}(\RE - 1)$, we  have $\|\myGApp-\myGPhi\| \leq c \sqrt{1+\frac{\delta_x}{\delta_\xi}} \frac{\delta_x}{\delta_\xi} $ for some constant $c>0$ independent of $\delta_\xi$ and $\delta_x$ that remains bounded for $R$ bounded, $\RE$ bounded away from $1$ and $\delta_x$ bounded away from $\sqrt{2}(\RE - 1)$.  %
\end{lemma}
\begin{proof}
  Again, the proof follows in an analogous manner as \cref{lem:radon_image_discretization} but now, one has to estimate $\abs{v_{pq}\bigl( \Pi(x) \bigr) - v_{pq}(x)}$ for $p,q$ fixed,
  \[
    v_{pq}(x) = \frac1{x \inprod \vartheta_q^\perp + \RE} w_{\delta_\xi}\Bigl(\frac{x \inprod \vartheta_q R}{x \inprod \vartheta_q^\perp + \RE} - \xi_p\Bigr),
  \]
  $x \in \Omega$ and $\Pi(x)$ denoting the projection
  of $x$ onto the set of all pixel centers $x_{ij}$.
  For this purpose, observe that the Euclidean norm of the weak derivative of
  $v_{pq}$ 
  also obeys the bound
  $c' = \frac{(R+1)^2}{(\RE - 1 - \delta_x/\sqrt{2})^3}$ on $\Omega + B(0,\frac{\delta_x}{\sqrt{2}})$. %
  Pursuing the strategy of the proof of \cref{lem:radon_image_discretization}, since the projection error obeys $\abs{\Pi(x) - x} \leq \frac{\delta_x}{\sqrt{2}}$, the area of the stripe $\abs{\frac{x \inprod \vartheta R}{x\inprod \vartheta^\perp + \RE} - \xi_p} \leq \delta_\xi$ enlarged by a ball of radius $\frac{\delta_x}{\sqrt{2}}$ within $\Omega + B(0,\frac{\delta_x}{\sqrt{2}})$ has to be estimated. However, such an estimate is, for instance, given by
  \[
    R\delta_\xi + \bigl(\sqrt{(\xi_p - \delta_\xi)^2 + R^2} + \sqrt{(\xi_p + \delta_\xi)^2 + R^2} + 2\delta_\xi \bigr) \frac{\delta_x}{\sqrt{2}} + \frac{3\pi}{4} \delta_x^2,
  \]
  which can, in turn, be estimated by $c''(\delta_\xi + \delta_x)$ for $c'' > 0$ which only depends on $R$, $W$ and the bound $\sqrt{2}(\RE - 1)$ on $\delta_x$.
  Following the proof of \cref{lem:radon_image_discretization}, one obtains
  \[
    \int_\Omega \abs{v_{pq}(\Pi(x)) - v_{pq}(x)} \dd{x} \leq c(\delta_\xi + \delta_x) \delta_x
  \]
  for a suitable $c > 0$. Further, for fixed $x \in \Omega$,
  \[
    \sum_{p=1}^P \abs{v_{pq}(\Pi(x)) - v_{pq}(x)} \leq \sqrt{2} c' \delta_x
  \]
  since the number of $p$ for which the weak derivative of $v_{pq}$ does not vanish in $x$ is still at most $2$. The latter two estimates suffice to carry out the proof analogous to \cref{lem:radon_image_discretization}, leading to
  the desired estimate after possibly adjusting $c$. This constant can in particular be chosen bounded under the stated conditions.
\end{proof}

As the final step, we estimate the error between the operators $\mM$ and $\mM_{\delta_\xi}$.

\begin{lemma}
  \label{lem:fanbeam_mult_discretization}
  We have $\norm{\mM - \mM_{\delta_\xi}} \leq c \delta_{\xi}$ where
  $c > 0$ stays bounded whenever $R$ stays bounded and $\RE$ is
  bounded away from $1$.
\end{lemma}

\begin{proof}
  The function $\xi \mapsto \sqrt{\xi^2 + R^2}$ is Lipschitz continuous on
  ${]{-\frac{\DW}2,\frac{\DW}2}[}$ with
  constant bounded by $c = \frac{W}{\sqrt{W^2 + 4R^2}}$ such that for $\xi \in \Xi_p$, we
  obtain the estimate
  $\abs{\sqrt{\xi^2 + R^2} - \sqrt{\xi_p^2 + R^2}} \leq c \delta_{\xi}$. Thus,
  \[
\norm{\mM - \mM_{\delta_\xi}} = \sup_{\xi \in {]{-\frac{\DW}2,\frac{\DW}2}[}} \ \Bigabs{ \sum_{p=1}^P \Bigl( \sqrt{\xi^2 + R^2} -
  \sqrt{\xi_p^2 + R^2} \Bigr) \chi_{\Xi_p}(\xi) } \leq c \delta_\xi. \] Since $W$ is bounded under the stated conditions, $c$ also remains bounded.
\end{proof}

Putting everything together allows us to derive convergence results
for the approximate fanbeam transform $\FanbeamApp$ towards $\Fanbeam$
as well as for their respective adjoints.
\begin{theorem}\label{thm:fanbeam_convergence}
  Let $\delta_\xi \to 0$ and $\frac{\delta_x}{\delta_\xi}\to 0$ and $\frac{\delta_\alpha}{\delta_\xi}\to 0$. Then, $\|\Fanbeam-\FanbeamApp\|\to 0$ and $\|\Fanbeam^*-(\FanbeamApp)^*\|\to 0$.
  If, additionally, $\delta_\alpha=\kronO(\delta_\xi^{1+\epsilon})$ and $\delta_x=\kronO(\delta_\xi^{1+\epsilon})$ for $\epsilon \in {]{0,\frac{1}{2}}[}$, then 
$\|\Fanbeam-\FanbeamApp\|=\kronO(\delta^\epsilon_\xi)$ and $\|\Fanbeam^*-(\FanbeamApp)^*\|=\kronO(\delta^\epsilon_\xi)$ where $\delta_\xi \leq \frac43 (\sqrt{2} - 1)$ and $\delta_x < \sqrt{2}(\RE - 1)$.
\end{theorem}
\begin{proof}
  Combining Lemmas~\ref{lem:F_detector_discretization}, \ref{lem:F_detector_angular_discretization} and  \ref{lem:fanbeam_image_discretization} analogously to  the proof of \cref{thm:radon_convergence} yields 
  $\|\myG-\myGApp\| \to 0$ and with the rate $\kronO(\delta_\xi^\epsilon)$
  in case the additional assumptions are satisfied, since $\sqrt{\delta_\xi} \abs{\log \delta_\xi}^{1/2} = \kronO(\delta_\xi^\epsilon)$ for $\epsilon \in {]{0,\frac{1}{2}}[}$. Now, as
\begin{align*}
  \|\Fanbeam-\FanbeamApp\|%
                          &\leq \|\mM_{\delta_\xi}\| \|\myG-\myGApp\| + \|\mM - \mM_{\delta_\xi}\| \|\myG\|
\end{align*}
and $\norm{\mM_{\delta_\xi}} \leq \frac12 \sqrt{W^2 + 4R^2}$, the convergence
to $0$ as well as the rate directly follow with \cref{lem:fanbeam_mult_discretization}. The statements for the adjoints are then immediate.
\end{proof}

\begin{remark}
  Note that many of the statements in \Cref{Subsec:Limited_information} concerning the Radon transform with incomplete angle information can be adapted
  to the fanbeam setting.
  For instance, the convergence results for the fanbeam transform
  can be extended to the limited angle setting of~\Cref{subsubsec:limited_angles}. %
  A transfer to the sparse-angle fanbeam transform is, however, not possible with the above techniques. We  nevertheless expect that a statement analogous to \cref{thm:finite_angle_convergence} is true.
\end{remark}

Faster convergence for functions with higher regularity analogous to \cref{thm:convergence_highregularity} can be shown.
\begin{theorem}
Let $f \in L^2(\Omega)$ and $g\in L^2(\Omega')$ with 
\begin{equation*}
\int_{\abs{t} \leq \delta_\xi} \modcont{\Fanbeam f}(t,0)^2 \dd{t} \leq
  c \delta_\xi^{1 + 2\epsilon} \norm{f}^2 \qquad \text{ and } \qquad\int_{\abs{t} \leq \delta_\xi} \modcont{g}(t,0)^2 \dd{t} \leq
  c \delta_\xi^{1 + 2\epsilon} \norm{g}^2
  \end{equation*}
   for  some $0 < \epsilon \leq 1$ and constant $c\geq 0$. 
  If additionally, $\delta_x=\kronO(\delta_\xi^{1+\epsilon})$ and \hbox{$\delta_\alpha=\kronO(\delta_\xi^{1+\epsilon})$}, then 
  \begin{equation}
  \| \FanbeamApp f - \Fanbeam f\|=\kronO(\delta_\xi^\epsilon) \qquad and \qquad \| (\FanbeamApp)^*  g- \Fanbeam^* g\| = \kronO(\delta_\xi^\epsilon).
  \end{equation}

\end{theorem}
\begin{proof}
As a consequence of \eqref{equ:proof_mod_cont_equivalence}, we have $\modcont{\myG f}(t,0)^2 \leq c' \bigl (t^2 \| \Fanbeam f\|^2+\modcont{\Fanbeam f}(t,0)^2\bigr)$.
With the first estimate in \eqref{equ:proof_modulus_to_estimate} it follows for some $c''>0$ that
\begin{equation} \label{equ:fanbeam_proof_estimate_higher_regularity}
\|\myG f-\myGS f\|^2  \leq \frac{1}{\delta_\xi} \int_{|h|\leq \frac{3}{2}\delta_\xi}  \modcont{\myG f}(h,0)^2 \dd{h}\leq \frac{c''}{\delta_\xi}\int_{|t|\leq \frac{3}{2}\delta_\xi} t^2+\modcont{\Fanbeam f}(t,0)^2 \dd{t}= \kronO(\delta_\xi^{2 \epsilon}).
\end{equation}
The rest of the proof works out completely analogously to \cref{thm:fanbeam_convergence}: By combining \eqref{equ:fanbeam_proof_estimate_higher_regularity}, Lemmas~\ref{lem:F_detector_angular_discretization}, \ref{lem:fanbeam_image_discretization} and the choice of the discretization parameters, one obtains the estimate $\|\myG f- \myGApp f\| \leq c \delta_\xi^\epsilon$. As in the proof of \cref{thm:fanbeam_convergence}, the rate of convergence of $\|\FanbeamApp f- \Fanbeam f\|$ then follows immediately.
Analogous arguments together with \cref{lem:fanbeam_detector_discretization_Backprojection} yield the result for the adjoint operator.
\end{proof}

\begin{remark}
  The restriction $\epsilon<\frac{1}{2}$ 
  appears in \cref{thm:fanbeam_convergence} since
  \cref{lem:fanbeam_mod_continuity} only yields this estimate for such $\epsilon$, but for more regular functions, this restriction can be removed.
  Due to the factorizations $\Fanbeam = \mM \myG$
  and $\FanbeamApp = \mM_{\delta_\xi} \myGApp$ and the sharpness of \cref{lem:fanbeam_mult_discretization}, the rate cannot improve beyond $\epsilon = 1$ using the presented strategy.
\end{remark}

\section{Numerical implementation and experiments}\label{Sec:Numeric}
In this section, we complement the previously discussed analytical results with concrete numerical considerations and experiments. 
\subsection{Numerical implementation}
We consider the algorithmic implementation of the Radon transform of a discrete function $f\in U \widehat =\RR^{N\times M}$ according to \eqref{equ:def_discrete_image_sinogram_spaces} with $f_{ij}$ the value associated to $X_{ij}$. Our considerations focus on the Radon transform as the pixel-driven backprojection is numerically well understood, see, e.g.,~\cite{Xie_CUDA_paralelization_2015}.  A major part in the computation of the discrete transformation $[\RadonDisc f]_{pq}$ consists of finding   all $x_{ij}$ such that the $w_{\delta_s}$ terms in \eqref{equ:def_discrete_radon} do not vanish, i.e., the $x_{ij}$ sufficiently close to $L(s_p,\varphi_q)$. In this case, we say that $x_{ij}$ is adjacent to $L(s_p,\vartheta_q)$. %
For fixed $p$ and $q$, the adjacency $x_{ij}$ can be determined  by the following steps: Introduce the $\bold x$ and $\bold y$ coordinates of the discrete grid, i.e., $\bold x_i=\delta_x\bigl(i-(N+1)/2\bigr)$ and $\bold y_j=\delta_x \bigl(j-(M+1)/2\bigr)$ for $i\in \{1,\dots,N\}$, $j\in \{1,\dots,M\}$. Fix $\bold y_j$ and compute $\vartheta_{\bold x}^{-1} ( \tilde s -\vartheta_{\bold y }\bold y_j)$ for $\tilde s\in \{s_{p-1},s_{p+1}\}$ if $\vartheta_{\bold x}^{-1}\neq 0$ for $(\vartheta_{\bold x},\vartheta_{\bold y})$ the components of $\vartheta_q$. These values establish the boundaries for the range of %
all $\bold x_i$ for which $(\bold x_i,\bold y_j)$ is adjacent to $L(s_p,\varphi_q)$. The computation of the Radon transform can thus be summarized in \cref{algo:Pixel-Driven-Radon}. Conversely, when considering $[\RadonDisc^*g]_{ij}$, one has to determine, for $x_{ij}$ and fixed $\varphi_q$, all $s_p$ such that $x_{ij}$ is adjacent to $L(s_p,\varphi_q)$, and then compute the corresponding weighted sum, see \cref{algo:Pixel-Driven-Back}.

\begin{algorithm}
    \caption{Pixel-driven Radon transform}
    \label{algo:Pixel-Driven-Radon}
    \begin{algorithmic}[1] %
        \Function{Radon$(f,N,M,P,Q,\delta_x,\delta_s,\{\varphi_q\}_q)$}{}
            \For{ $p\in \{1,\dots,P\}$, $q \in \{1, \dots,Q\}$} \label{line:outer_forloop}
				\State $\text{val} \gets 0$                
                \State  $(\vartheta_{\bold x},\vartheta_{\bold y})\gets \vartheta_q$
				 \For{ $j\in \{1,\dots,M\}$} \label{line:middle_forloop}
                	\If {$\vartheta_{\bold x}\neq 0$}
                		\State $(\underline{\bold x},\overline{\bold x}) \gets \text{sort}\bigl(\vartheta_{\bold x}^{-1} ( s_{p-1} -\vartheta_{\bold y }\bold y_j), \vartheta_{\bold x}^{-1} ( s_{p+1} -\vartheta_{\bold y }\bold y_j)\bigr)$            		
						\State $\mathcal{X}_{pq}^j \gets\set{i\in \{1,\dots,N\}}{\bold x_i\in [\underline{\bold x},\overline{\bold x}]}$
                	\ElsIf{$\abs{\vartheta_{\bold y}\cdot \bold y_j-s_p}<\delta_s$}
                	 	\State $\mathcal{X}_{pq}^j \gets \{1,\dots,N\}$ 
                	 \Else
	                	 \State $\mathcal{X}_{pq}^j \gets \emptyset$ 
					\EndIf
					\For {$i \in \mathcal{X}_{pq}^j$} \label{line:inner_forloop}
					\State $\alpha\gets \delta_s-\abs{s_p-\vartheta_{\bold x} \bold x_i-\vartheta_{\bold y} \bold y_j}$
					\State $\text{val}\gets \text{val}+\alpha f_{ij}$
					\EndFor
				\EndFor				
			\State $g_{pq} \gets \frac {\delta_x^2}{\delta_s^2}\text{val}$
            \EndFor
            \State \textbf{return} $g$
        \EndFunction
    \end{algorithmic}
\end{algorithm}

\begin{algorithm}
    \caption{Pixel-driven backprojection}
    \label{algo:Pixel-Driven-Back}
    \begin{algorithmic}[1] %
        \Function{Backprojection$(g,N,M,P,Q,\delta_x,\delta_s,\{\varphi_q\}_q)$}{}
            \For{ $i\in \{1,\dots,N\}$, $j\in \{1,\dots,M\}$}
				\State $\text{val}\gets 0$            	
            	\For{ $q\in \{1,\dots,Q\}$}      						            	
					\State $ p\gets P (\vartheta_q \cdot x_{ij}+\frac{1}{2})$
					\State $ (\overline{p},\underline{p}) \gets \bigl(\text{ceil}(p),\text{floor}(p)\bigr)$
					\State $\alpha  \gets  \overline{p} - p $
					\State $\text{val}\gets \text{val} +\Delta_q \big(\alpha g_{\underline{p}q}+(1-\alpha) g_{\overline{p}q}\big)$
				\EndFor				
			\State $f_{ij} \gets \frac {\text{val}}{\delta_s}$
            \EndFor
            \State \textbf{return} $f$
        \EndFunction
    \end{algorithmic}
\end{algorithm}

\begin{remark}
 Note that the determination of $\overline{\bold x},\underline{\bold x}$ in \cref{algo:Pixel-Driven-Radon} becomes unstable for small $\vartheta_{\bold x}$, which can be remedied by relaxing the condition $\vartheta_{\bold x} \neq 0$ to $|\vartheta_{\bold x}|>\epsilon$ for some $\epsilon>0$ or by splitting into almost horizontal and almost vertical lines as done, for instance, in \cite{Averbuch01fastslant}. 
\end{remark}

Due to high dimensionality,  the weights used in the computation of $\RadonDisc$ are typically not saved but rather computed on the fly.
Note that all operations inside the for loop in Line \ref{line:outer_forloop} of \cref{algo:Pixel-Driven-Radon} can be parallelized with one thread for each sinogram pixel $S_p\times\Phi_q$  without creating race conditions, since the computation $[\RadonDisc f]_{pq}$ is independent of $[\RadonDisc f]_{\tilde p \tilde q}$ for $(p,q)\neq (\tilde p, \tilde{q})$. In \cite{Du2017}, further parallelization using a thread per pixel was used, which however introduces possible race conditions and thus requires further considerations.  

Let us briefly discuss the complexity of the algorithm.
In the following we assume that all the basic operations in  \cref{algo:Pixel-Driven-Radon} and \cref{algo:Pixel-Driven-Back} possess the same computational complexity and hence, the total computational complexity can be estimated by the number of executed operations.

It is easy to see that the determination of all sets $\mathcal{X}_{pq}^j$ with $j=1,\dots,M$ for fixed $p$ and $q$ requires $\kronO(M)$ operations, so the total effort in the determination of the adjacency relation per projection is $\kronO(M P)$ and thus, $\kronO(MPQ)$ for the entire sinogram. %
Since each pixel $x_{ij}$ is adjacent to at most two detector offsets, we get  $\sum_{p=1}^P \sum_{j=1}^M |\mathcal X_{pq}^j|\leq 2NM$ for fixed $q$, and thus, the weighted summation in the computation of a projection of the pixel-driven Radon transform requires $\kronO(N M)$ operations. For the computation of the entire sinogram this leads to $\kronO(NMQ)$ operations. In total, we obtain
the complexity estimate $\kronO(MPQ + NMQ)$.

 When assuming $N\sim P$, the total computational complexity of the computation of the Radon transform is $\kronO(NMQ)$, which is the same as for all other common discretization approaches for the Radon transform. For the backprojection it is easy to see, that the effort for computation of $[\RadonDisc^* g]_{ij}$ for fixed $i,j$ is $\kronO(Q)$  and thus computing the entire backprojection requires $\kronO(NMQ)$ operations, again in line with other common approaches.

\begin{remark}
  Let us estimate the number of adjacent pixels for a single detector, as this is the relevant quantity concerning complexity if the computation can be parallelized as mentioned above. All adjacent pixels have to be contained in a stripe (slightly larger than the stripe associated with the detector), whose area is not greater than $\sqrt{N^2+M^2}\delta_x (\sqrt{2}\delta_x + 2\delta_s)$. Dividing
  by the pixel area $\delta_x^2$ yields an upper bound for
  the number of adjacent pixels.
Conversely, if $\varphi_q = 0$ or $\varphi_q = \frac\pi2$, the maximal number of adjacent pixels can usually be estimated from below by $\max(N,M) \lceil \frac{2\delta_s}{\delta_x} \rceil$.
Hence, in the worst case, the number of adjacent pixels lies in the range $[\max(N,M) \lceil \frac{2\delta_s}{\delta_x}\rceil, \sqrt{N^2+M^2}(\sqrt{2} + \frac{2\delta_s}{\delta_x})]$.
Thus, in the standard setting $N=M=P$, $\delta_x=\frac{2}{N}$ and $\delta_s=\frac{2}{P}$, %
this is contained in the range $[2N,5N]$.
In comparison, the number of pixels relevant for the computation of a single detector via ray-driven methods is, in the worst case, in the range $[2N,3N]$,
depending on the implementation. In this light, ray-driven and pixel-driven methods approximately share the same complexity.

In the case $P\ll N$ associated to the obtained convergence results, choosing $N\sim M$,  $\delta_x=\frac{2}{N}$ and $\delta_s=\frac{2}{P}$, the number of relevant pixels for the computation of a single detector via the pixel-driven method is roughly of $\mathcal{O}(\frac{N^2}{P})$. Note that for ray-driven methods, the case $P\ll N$ is not feasible, as only a fraction of the available pixels are used in the computation of a projection.
\end{remark}

In the supplementing information of \cite{C8NR09058K} it is shown, that the pixel-driven method (there referenced as proposed method) for $P=N$ can be executed at comparable speed as other discretization approaches.

\subsection{Numerical examples}

In this section, we study results for the aforedescribed implementation on a concrete example.

We start by considering the (modified) Shepp--Logan phantom \cite{1974ITNS...21...21S} in terms of qualitative results, showing that indeed, suitable approximations can be obtained.
Figure \ref{Fig:Shepp_logan} depicts the Shepp--Logan phantom and its discrete Radon transform via the pixel-driven approach as well as ray-driven approach (the latter computed using the ASTRA toolbox \cite{Palenstijn2016,vanAarle:16}), and the corresponding pixel-driven backprojection, where the phantom has $1200\times 1200$ pixels and the sinogram has $1200\times 360$ pixels with angles uniformly distributed in $[0,\pi[$. This standard example shows that qualitatively, the proposed discretization indeed yields suitable results visually almost indistinguishable from the ray-driven transform. 
What is visually difficult to see --- due to the high number of angles used --- is that in some projections, oscillations are created, see \cref{Fig:Shepp_logan_angle}. There, one can see that for $\varphi=45^\circ$, strong oscillations occur with increasing amplitude for higher resolutions, while for the exemplary angle $\varphi=42.5^\circ$, only very mild oscillations occur. For small resolutions, the pixel-driven projection is quite similar to the ray-driven projection.

{
\begin{figure}[t]
\center{
\begin{overpic}[height=0.23\textheight]{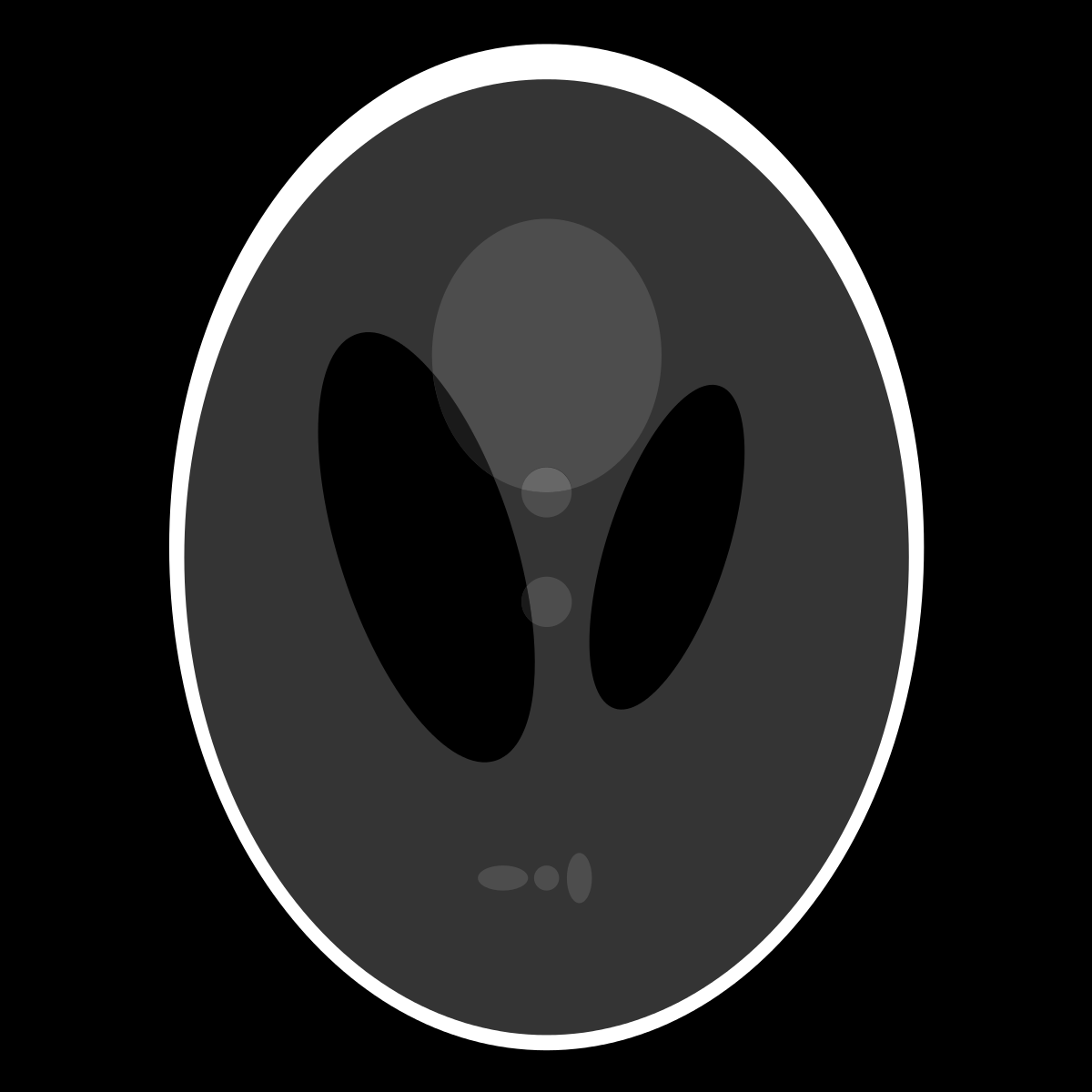}
\put (5,88) {}
\end{overpic}
\begin{overpic}[height=0.23\textheight]{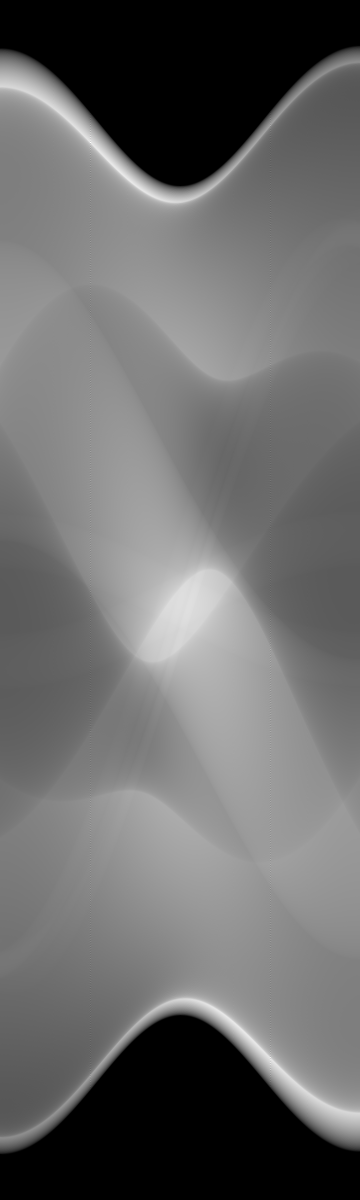}
\put (10,88) {}
\end{overpic}
\begin{overpic}[height=0.23\textheight]{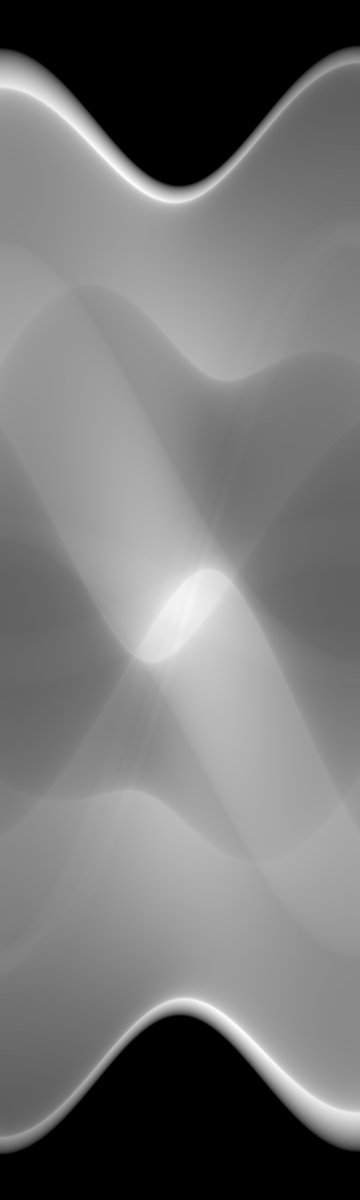}
\put (10,88) {}
\end{overpic}
\begin{overpic}[height=0.23\textheight]{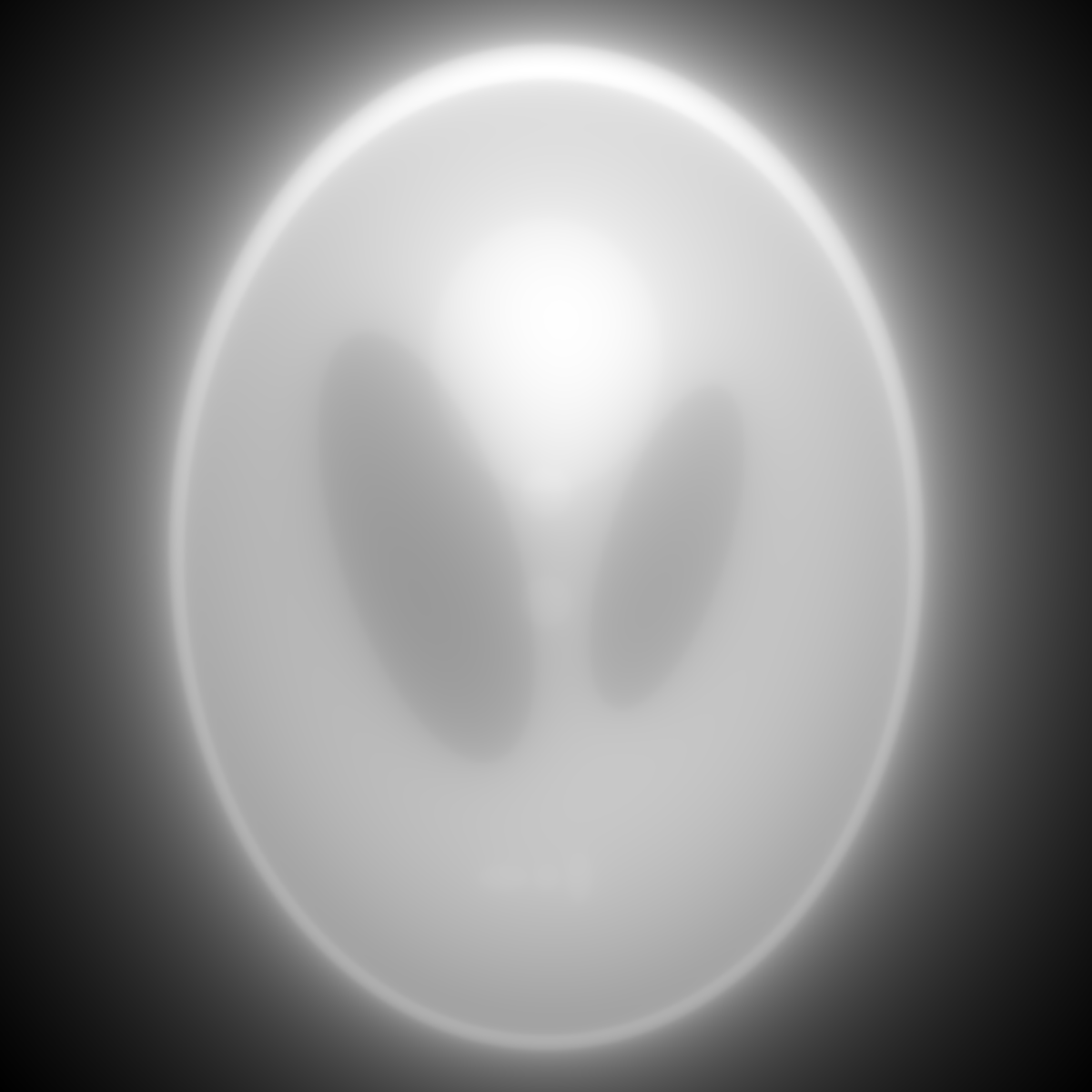} 
\put (5,88) {}
\end{overpic}
\caption{From left to right: The (modified) Shepp--Logan phantom ($1200\times1200$ pixels), its  pixel-driven Radon transform, its ray-driven Radon transform (both $360\times 1200$ pixels) and the pixel-driven backprojection.}
\label{Fig:Shepp_logan}
}
\end{figure}
}

\begin{figure}[t]
\center{
\includegraphics[scale=0.73]{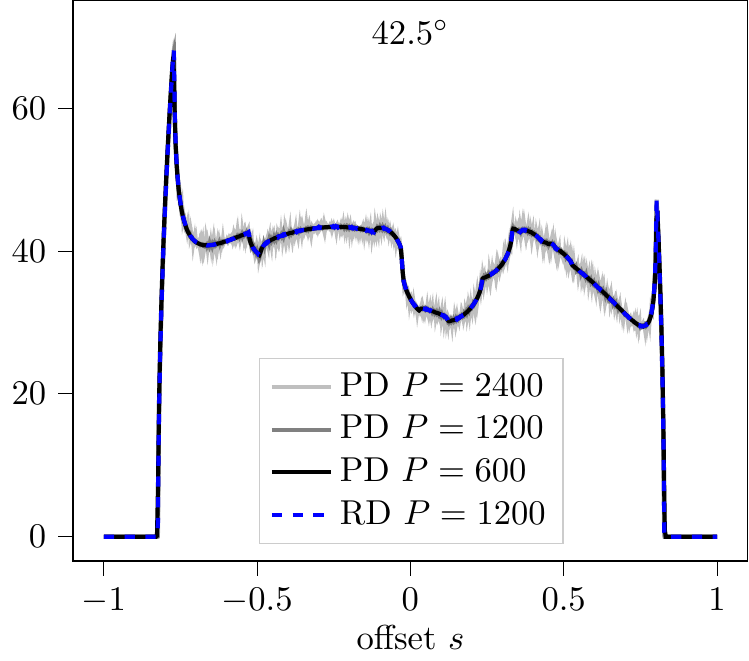}%
\includegraphics[scale=0.73]{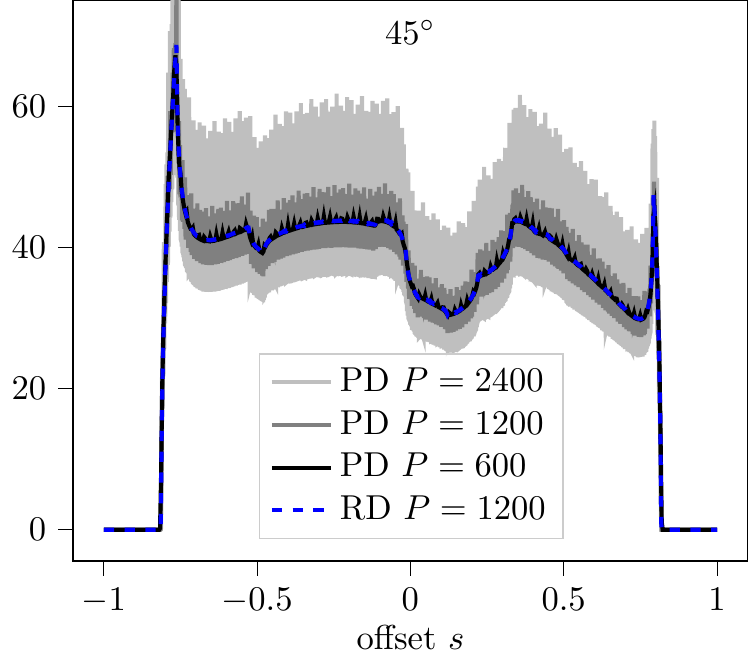}%
\caption{Projection plots for the  Shepp--Logan phantom for $\varphi\in \{42.5^\circ,45^\circ\}$ with the pixel-driven method (PD) for $P\in \{2 N,N,\frac{1}{2}N\}$ as well as the ray-driven projection (RD) for $P=N$.} 
\label{Fig:Shepp_logan_angle}
}
\end{figure}

\begin{figure}
\center
\includegraphics[scale=1]{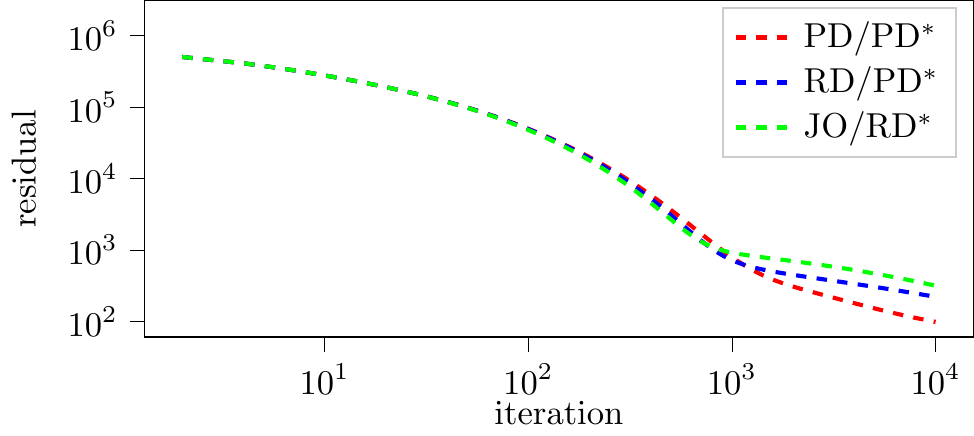}
\caption{Log-log plot of the residuals of the Landweber iteration using either the adjoint PD/PD$^*$ method  with pixel-driven transform and backprojection, the non-adjoint  RD/PD$^*$ method with ray-driven transform and pixel-driven backprojection and the JO/PD$^*$ method with Joseph interpolation Radon transform and pixel-driven backprojection on the Shepp--Logan phantom with $N=300$, $P=300$ and $Q=100$.}
\label{Fig:Shepp_logan_log_log}
\end{figure}

Furthermore, we tested the Landweber iteration \cite{1951AmJM...73..615L} for Radon inversion, i.e., the solution of $\Radon f=g$, for different discretization strategies. The classical Landweber iteration requires both a discrete forward operator and a discrete adjoint, however, as mentioned earlier, in tomography, these discrete operators are in fact often not adjoint due to discretization errors. Since the convergence theory of the Landweber method is based on adjoint operators, it is not quite clear from a theoretical perspective how ``non-adjoint'' methods behave.
We consider three versions of the Landweber iteration: the adjoint method ``$\text{PD/PD}^*$'' using pixel-driven transform and backprojection, the non-adjoint  method \hbox{``$\text{RD/PD}^*$''} using ray-driven forward and pixel-driven backprojection  and the non-adjoint method ``$\text{JO/PD}^*$'' using Joseph interpolation kernel \cite{4307572} for the Radon transform, and pixel-driven backprojection. (We employ the ASTRA toolbox for the latter two methods' forward and backprojections, see \cite{Easy_implementation}, and all methods executed with single precision.) The data on the right-hand side are created by the respective forward operators.  
One might expect that the non-adjointness has a negative effect on the convergence speed of the iterative method.
Indeed, \cref{Fig:Shepp_logan_log_log} depicts the resulting $L^2$ residuals in a log-log plot, showing that initially, the residuals behave almost  identically, but the non-adjoint ASTRA methods slow down significantly at some point, while the adjoint method's residual continues to decrease at a higher rate.  This suggests that the Landweber iteration suffers from worse convergence properties for non-adjoint methods. 
A similar experiment was already presented by the authors in the supplementing information of \cite{C8NR09058K}. 

\begin{remark}
For the simple case of the Landweber iteration, this experiment shows benefits of using adjoint discrete operators concerning convergence properties, which we believe to extend to other iterative solution methods. This suggests a theoretical advantage of adjoint methods, which, together with the gained knowledge regarding the pixel-driven method's convergence, makes it worthwhile studying.
\end{remark}

\subsection{Numerical convergence rates}
We assume in the following square images with $N\times N$ pixels representing the discretization of $[-1,1]^2$, and sinograms with $P\times Q$ pixels where the angles are uniformly distributed in $[0,\pi]$.
Further, we denote by $\delta =(\delta_s,\delta_\varphi,\delta_x)=$\hbox{$(\frac{2}{P},\frac{\pi}{Q},\frac{2}{N})$} the degrees of discretization and in particular recall that then, $N$, $P$ and $Q$ denote the square root of the number of pixels, the number of detectors and the number of angles, respectively. 

In this subsection we focus on the impact of different strategies concerning the choice of the discretization parameters onto the degree of approximation for a very simple example. 
We consider the function
\begin{equation}
  \label{eq:numerical_test}
f(x)= \chi_{B(0,r)}(x) \qquad \text{with} \qquad [\Radon f](s,\varphi)= g(s,\varphi)=
\begin{cases}
  \sqrt{r^2-s^2} & \text{if} \ \abs{s} \leq r, \\
  0 &\text{else},
\end{cases}
\end{equation}
 where $r=0.6$ and in particular, the transformed function does not depend on $\varphi$ as $f$ is rotationally invariant. 
The discrete Radon transform via \eqref{equ:def_Radon_all_discrete} applied to the function $f$ with respect to the discretization $\delta$ is denoted by $g^\delta(s,\varphi)=\sum_{p=1}^{P} \sum_{q=1}^{Q} \chi_{S_p}(s) \chi_{\Phi_q}(\varphi) g^\delta_{pq}$.

To quantitatively compare the effect of the approximation we consider the $L^2$-error between the continuous and discrete Radon transform applied to $f$ whose square is computed via
\begin{equation*}
\int_{[{-\pi,\pi}[} \int_{-1}^1 |g-g^\delta|^2 \dd{s} \dd{\varphi}= \norm{g- g^\delta}^2=\norm{g}^2+ \norm{g^\delta}^2 - 2 \scp{g}{g^\delta}.
\end{equation*}
Due to the explicit form of $g$ and $ g ^\delta$, one computes
\begin{align*}
  &\norm{g}^2=\frac {8\pi} {3} r^3 =2\pi \sum_{p=1}^{P} \Bigl[  \Pi_{[-r,r]}(s)r^2-\frac{\Pi_{[-r,r]}(s)^3}{3} \Bigr]_{s=s_p-\frac{\delta_s}{2}}^{s_p+\frac{\delta_s} {2}},
    \ \ %
    \norm{g^\delta}^2  = \!\!\sum_{p,q=1}^{P,Q}\!\!  {\delta_s} \Delta_q| g^\delta_{pq}|^2,
\\ \notag
&\scp{g}{g^\delta} = \sum_{p,q=1}^{P,Q} \Delta_q g^\delta_{pq}  \int_{S_p}g(s,\varphi) \dd{s} = \sum_{p,q=1}^{P,Q} \Delta_q g^\delta_{pq}  \bigl(G(s_p+\tfrac{ \delta_s}{2})-G(s_p-\tfrac{ \delta_s}{2})\bigr),
\\
&G(s)=\frac{1}{2} \Bigl( \Pi_{[-r,r]}(s) \sqrt{r^2-\Pi_{[-r,r]}(s)^2}+ r^2\arcsin\Bigl( \frac{\Pi_{[-r,r]}(s)}{r}\Bigr)\Bigr),
\end{align*}
where $\Pi_{[-r,r]}$ is the projection onto $[-r,r]$, i.e., $\Pi_{[-r,r]}(s) = \min(r, \max(-r, s))$, and $G$ is an indefinite integral of $s \mapsto g(s,\varphi)$ for a $\varphi$. This approach can also be adapted in a straightforward way to measure the $L^2$-error
$\bigl(\int_{-1}^1 \abs{g(s,\varphi) - g^\delta(s,\varphi)}^2 \dd{s} \bigr)^{1/2}$
of the projection associated with a fixed angle $\varphi \in {[{-\pi,\pi}[}$.

Now, concerning the expected behavior of the $L^2$-error, it is possible to verify that  for all $\epsilon<1$, we have $\modcont{g}(h,0) = \kronO(\abs{h}^{\epsilon})$.
Hence, \cref{thm:convergence_highregularity} and \cref{thm:finite_angle_convergence} guarantee convergence rates of $\kronO(\delta_s^{\epsilon})$ for each $\epsilon < 1$ when choosing $\delta_x = \kronO(\delta_s^2)$ and $\delta_\varphi = \kronO(\delta_s^2)$ for both the discrete Radon transform as well as the sparse angle transform, while the choice $\delta_x \sim \delta_s$, $\delta_\varphi \sim \delta_s$ does not guarantee convergence.
For this reason, we perform experiments for both choices.

\begin{figure}
\includegraphics[scale=0.73]{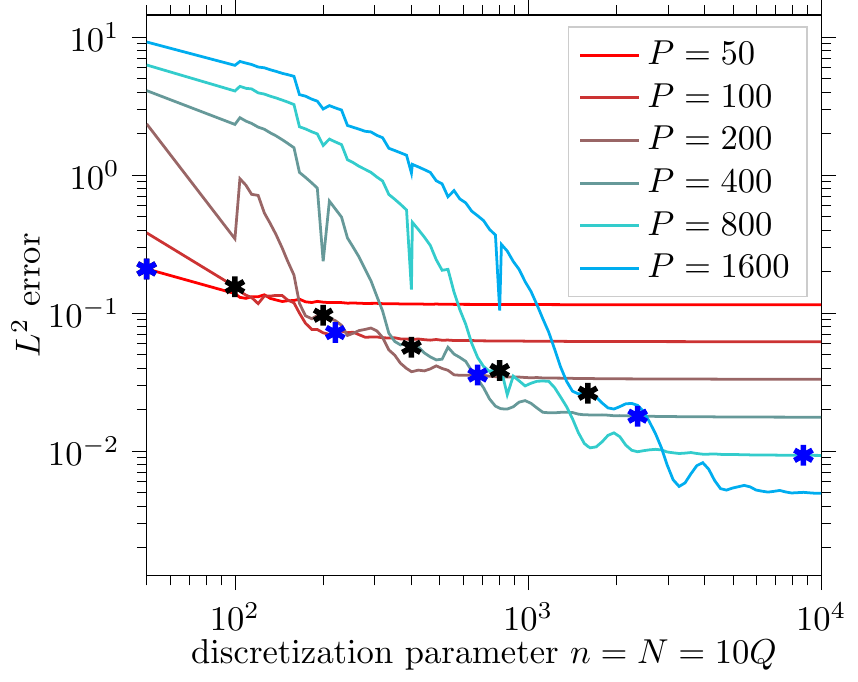}%
\includegraphics[scale=0.73]{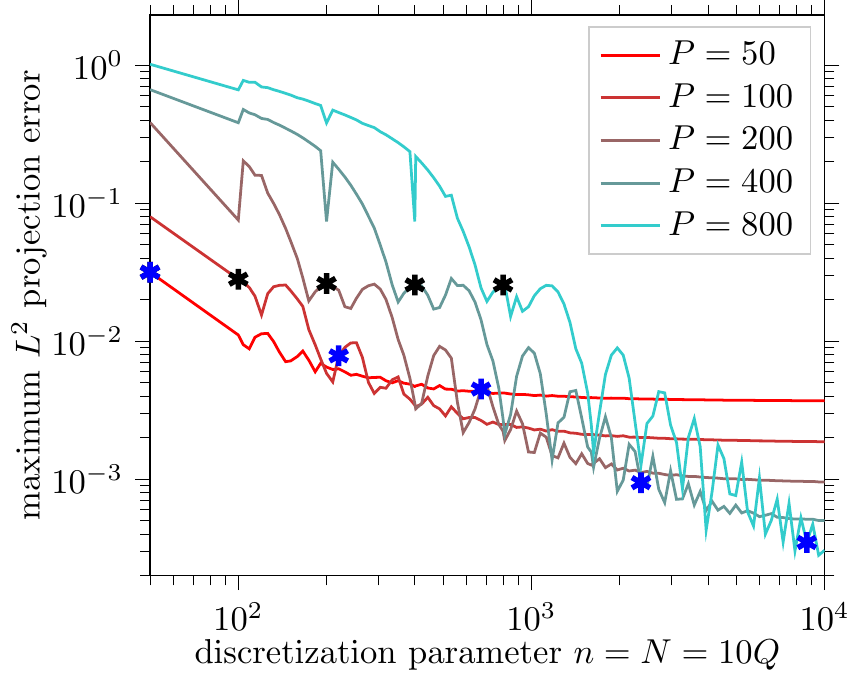}
\caption{Log-log plots of the $L^2$-errors for $P$ detectors where $P \in \sett{50,100,200,400,800,1600}$ (left: $L^2$-error of the sinogram, right: maximal $L^2$-error of each projection). The discretization level $\nchoice = N = 10Q$ is plotted against the respective $L^2$-error.
  The black and blue asterisks mark the errors for the choices $\nchoice = P$ and $\nchoice =\frac{ P^2}{90}+P$, respectively.} %
\label{Fig:Log-log_Error}
\label{Fig:Log-log_Error_Worst_angle}
\end{figure}

\Cref{Fig:Log-log_Error} shows log-log plots of the $L^2$-errors for~\eqref{eq:numerical_test}, where the $L^2$-error with respect to the whole sinogram domain %
and the maximal $L^2$-error of a single projection with respect to each discrete angle is plotted.  Each plot corresponds to a fixed $P$ %
and varying $\nchoice$ such that %
$\nchoice = N$ and $Q = \frac{\nchoice}{10}$. One can see that there is always a point where increasing $\nchoice$ does no longer reduce the error, i.e., where the maximal accuracy that is possible for fixed $P$ is reached. In \cref{Fig:Log-log_Error}, we also mark both the choice
$\nchoice \sim P^2$ and $\nchoice \sim P$ on the plots. One can see that indeed, as predicted by the theory, in case of $\nchoice \sim P^2$, both the $L^2$-error on the whole sinogram domain as well as the maximal $L^2$-error of each projection vanish with some rate that can be identified to roughly correspond to $\kronO(\delta_s)$, which indeed appears to the best convergence rate in this scenario. For the choice $\nchoice \sim P$, convergence is not guaranteed, however, the $L^2$-error on the sinogram domain still seems to vanish with some rate, presumably since the data $f$ according to~\eqref{eq:numerical_test} does not reflect the worst case. In contrast, the maximal $L^2$-error of each projection apparently does not vanish, i.e., not satisfying the convergence assumption does indeed lead to non-convergence.

\begin{figure}
  \includegraphics[scale=0.73]{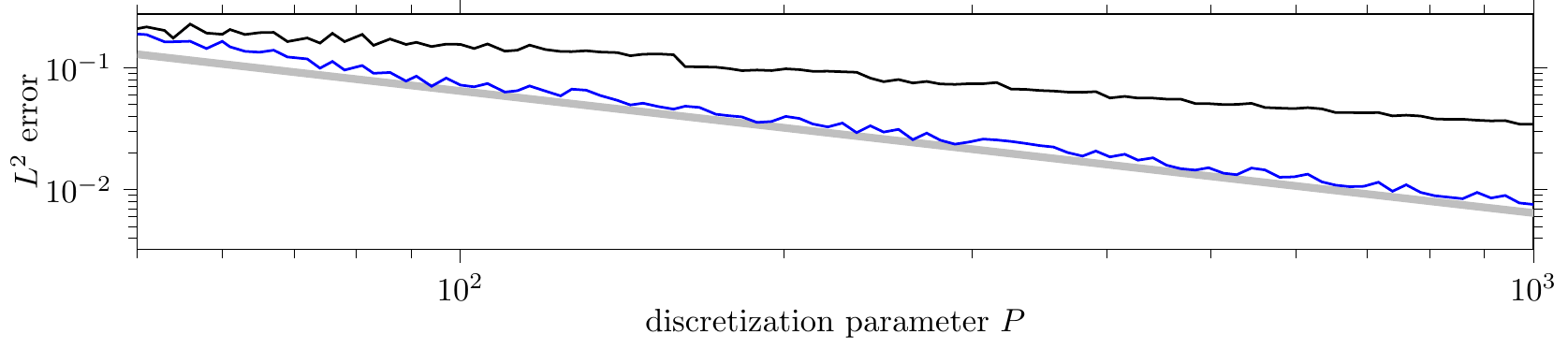}
\caption{Log-log plot of the $L^2$-error on the whole sinogram domain against the discretization level $P$ for $N = P$, $Q = \frac{P}{10}$ (black) and $N = \frac{P^2}{90}+P$, $Q = \frac{ P^2}{900}+\frac{P}{10}$ (blue). %
  The gray line represents the convergence rate $\kronO(\delta_s)$.}
\label{Fig:Log-log-dependent}
\end{figure}

\begin{figure}
\newcommand{\mylen}{0.65}
\begin{tabularx}{\textwidth}{c@{\,}c@{\,}c}
  \includegraphics[scale=\mylen]{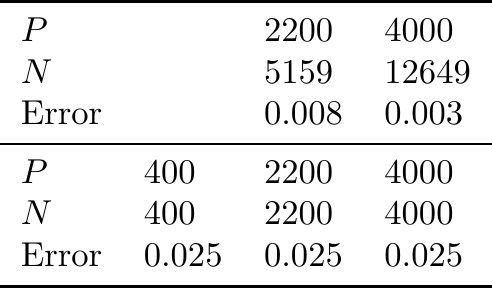}
  &
\includegraphics[scale=\mylen]{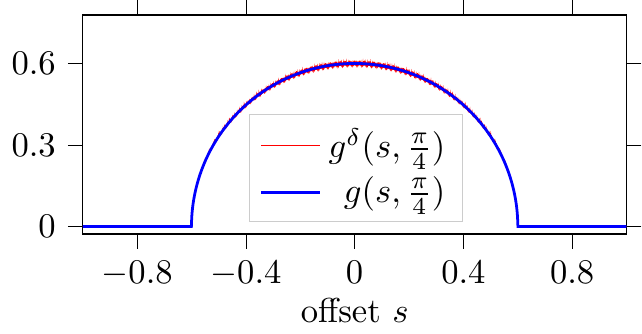}  
& \includegraphics[scale=\mylen]{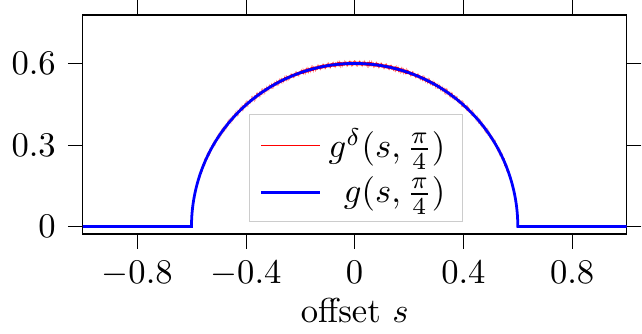} 
\\ 
\includegraphics[scale=\mylen]{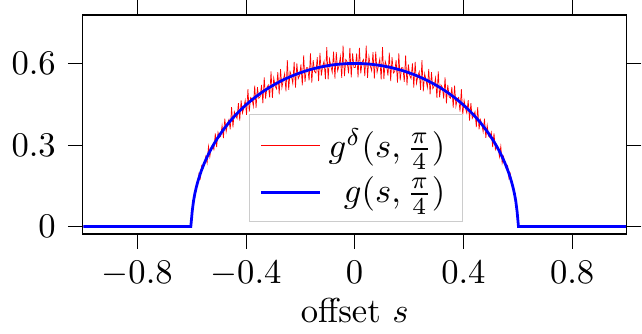} 
& \includegraphics[scale=\mylen]{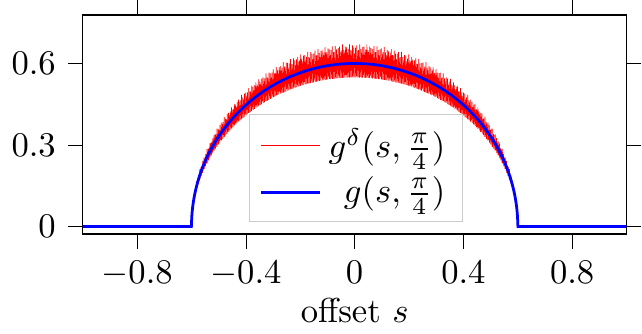}  & \includegraphics[scale=\mylen]{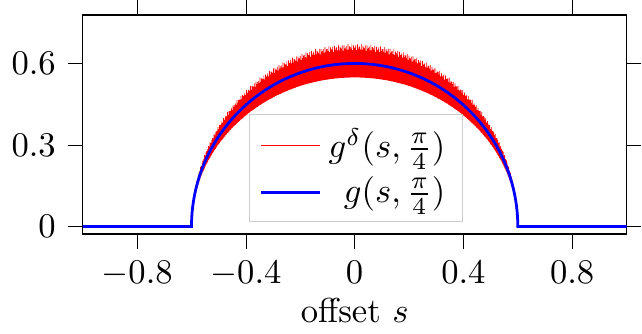} 
\end{tabularx}
\caption{Comparison plots for the continuous projection (blue) and the discrete projection corresponding to the maximal $L^2$-error (red) for convergent and non-convergent
  discretization parameter choice.
  The table summarizes the choice of $P$ and $N$ as well as the resulting error. The top row corresponds to the choice $N \sim P^{3/2}$ while the bottom row corresponds to the choice $N = P$. 
}
\label{Fig:Worst_angle_oscilations}
\end{figure}

These observations can also be confirmed by examining the $L^2$-error on the whole sinogram domain in dependence of $P$ for both choices $\nchoice \sim P$ and $\nchoice \sim P^2$, see \cref{Fig:Log-log-dependent}, where this error is plotted against $P$ such that the convergence rates become apparent. Further, the non-convergence behavior for the maximal $L^2$-error of each projection is investigated in more detail in \cref{Fig:Worst_angle_oscilations}. There, comparison plots of the discrete projections corresponding to the maximal error (typically an angle that is an integer multiple of $\frac \pi 2$) are shown. In these plots, it becomes apparent that the error is dominated by high-frequency oscillations that remain constant for the choice $N \sim P$, but vanish, e.g., for the choice $N \sim P^{3/2}$. This confirms that with a suitable parameter choice rule, the unwanted oscillatory behavior can be suppressed.

\section{Conclusion and outlook}
\label{sec:conclusions}
This work presents a novel rigorous analysis of pixel-driven approximations of the Radon transform and the backprojection. It is shown that this scheme leads to convergence in the operator norm $L^2(B(0,1))\to L^2(\mathbb{R}\times S^1)$	subject to suitably chosen discretization parameters $\delta_s,\delta_\varphi,\delta_x$ such that the ratios of $\delta_x$ and $\delta_\varphi$ to $\delta_s$ vanish. Moreover, in case of $\delta_s\to 0$, $\frac{\delta_x}{\delta_s} =\kronO(\delta_s^{1+\epsilon})$ and $\frac{\delta_\varphi}{\delta_s} =\kronO(\delta_s^{1+\epsilon})$ with $0<\epsilon\leq \frac{1}{2}$, the rate $\kronO(\delta_s^\epsilon)$  in operator norm can be achieved. In particular, the analysis
ensures convergence for asymptotically smaller image pixels than detector pixels which is in contrast to the common choice of using the same
magnitude of discretization for detectors and image pixels. %
Furthermore, we obtain $L^2$-convergence for each
projection of %
the pixel-driven sparse-angle Radon transform,
given suitable parameter choice, and thus ensuring that high-frequency artifacts vanish in each projection.
The mathematical scheme and analysis is extended to the fanbeam transform with analogous convergence results, showing that the basic concept of the discretization framework is applicable to a larger class of projection operators. Future works might extend this mathematical understanding to other projection operations, such as the conebeam transform or three-dimensional Radon transform \cite{Natterer:2001:MCT:500773}. Further practical experiments and investigations will also be necessary to fully understand the accuracy of pixel-driven methods. 

 \bibliographystyle{siamplain}
 \bibliography{references}

\appendix
\section{Proof of  \texorpdfstring{\cref{lem:radon_angle_mod_continuity}}{Lemma 3.4}}
\label{app:proof}
In an analogous fashion to the proof of  \cref{lem:radon_mod_continuity}, we will estimate $\|\Radon^*\Radon - \Radon^*T_{0,\gamma}\Radon\|$, where $T_{0,\gamma}$ is a translation of the second argument by $\gamma \in \RR$. In order to do so, one computes for $f\in L^2(\Omega)$ and $x\in B(0,1)$, denoting by $A_\gamma x$ the rotation of $x$ by the angle $\gamma$,
\begin{align}
  [\Radon^*T_{0,\gamma}\Radon f] (x)&= \int_{[-\pi,\pi[} \int_{\RR}f\big( {(x \cdot \vartheta(\varphi))} \vartheta(\varphi+\gamma)+t \vartheta(\varphi+\gamma)^\perp \big)\dd{t} \dd{\varphi}
                                    \notag \\
                                  &= \int_{[-\pi,\pi[} \int_{\RR} f(A_\gamma x + t \vartheta(\varphi)^\perp) \dd{t} \dd{\varphi} %
                                      = \int_\Omega \frac{2}{|A_\gamma x-y|} f(y) \dd{y} \notag \\ &= 2 \int_\Omega k_\gamma(x,y)f(y)\dd{y}, \notag
\end{align}
where we used polar coordinates centered around $A_\gamma x$
and set
$k_\gamma(x,y)=\frac{1}{|A_\gamma x-y|}$. %
Arguing along the lines of Lemma~\ref{lem:radon_mod_continuity} and employing
the Cauchy--Schwarz estimate then leads to
\[
  \modcont{g}(0,\gamma)^2 \leq \norm{\Radon^*\Radon - \Radon^* T_{0,\gamma}\Radon}\norm{f}^2
  \quad\text{and}\quad
  \norm{\Radon^*\Radon - \Radon^* T_{0,\gamma} \Radon}
  \leq 2 \sqrt{M_1(\gamma) M_2(\gamma)},
\]
where
\[ M_1(\gamma) = \sup_{x \in \Omega} \int_\Omega \abs{k_0(x,y) - k_\gamma(x,y)} \dd{y}, \quad
  M_2(\gamma) = \sup_{y \in \Omega} \int_\Omega \abs{k_0(x,y) - k_\gamma(x,y)} \dd{x}.\]
In the following, we show that both $M_1(\gamma)$ and $M_2(\gamma)$ are $\kronO(\abs{\gamma \log(\abs{\gamma})})$ for $\abs{\gamma} \leq \frac{\pi}4$ which yields the claim.

Let us first estimate $M_1(\gamma)$. Fix $x \in \Omega$ and note
that for $y \in \Omega$ such that $\abs{x - y} \leq \abs{A_\gamma x - y}$,
we can estimate, using the triangle inequality and convexity of $t \mapsto t^{-1}$ on the positive axis,
\[
  \Bigabs{\frac1{\abs{x - y}} - \frac1{\abs{A_\gamma x - y}}}
  \leq
  \frac1{\abs{x - y}} - \frac1{\abs{x - y} + \abs{A_\gamma x - x}} \leq
  \frac{\abs{A_\gamma x - x}}{\abs{x - y}^2}.
\]
Now, $\Omega \subset B(x,2)$ such that with $d_x(\gamma) = \abs{A_\gamma x - x} \leq 2$, we get
\begin{align*}
  &\int_{\set{y \in \Omega}{\abs{x - y} \leq \abs{A_\gamma x - y}}}
    \abs{k_0(x,y) - k_\gamma(x,y)} \dd{y} \\
  &\qquad
  \leq \int_{B(x,d_x(\gamma))} \frac1{\abs{x-y}} \dd{y} +
  \int_{B(x,2) \setminus B(x,d_x(\gamma))}
  \frac{d_x(\gamma)}{\abs{x-y}^2} \dd{y} \\
  &\qquad
  \leq 2\pi d_x(\gamma)\bigl( 1 + \log(2) - \log(d_x(\gamma)) \bigr)
  = - 2\pi d_x(\gamma) \log\Bigl( \frac{d_x(\gamma)}{2\eulerE} \Bigr).
\end{align*}
The integral on the set of $y \in \Omega$ where $\abs{A_\gamma x - y} \leq \abs{x - y}$ can be estimated analogously with the same estimate. We have
$d_x(\gamma) \leq \sqrt{2(1 - \cos(\gamma))} \leq \abs{\gamma}$, such that
for $\abs{\gamma} \leq 1$, by monotonicity of $t \mapsto -t\log(t/(2\eulerE))$ on ${]{0,2}]}$, one obtains
\[
  M_1(\gamma) \leq \sup_{x \in \Omega} -4\pi d_x(\gamma) \log\Bigl(\frac{d_x(\gamma)}{2\eulerE}\Bigr)
  \leq 4\pi \Bigabs{\gamma \log\Bigl(\frac{\abs{\gamma}}{2\eulerE} \Bigr)} = 4\pi\abs{\gamma}\bigl(1 + \log(2) +
  \abs{\log(\abs{\gamma})} \bigr).
\]
Further restricting $\abs{\gamma} \leq \frac{\pi}{4} < 1$ gives $1 + \log(2) \leq c_0 \abs{\log(\abs{\gamma})}$ for some $c_0 > 0$ independent of $\gamma$, so we finally obtain
$M_1(\gamma) \leq c \abs{\gamma \log(\abs{\gamma})}$ for some $c > 0$.

For the remaining estimate of $M_2(\gamma)$, note that $|A_\gamma x-y|=|x-A_{-\gamma}y|$ since rotations leave norms unchanged. Therefore, $k_\gamma(x,y)
= k_{-\gamma}(y,x)$ and consequently, $M_2(\gamma) = M_1(-\gamma)$, so the claimed rate follows immediately.
\hfill\proofbox

\end{document}